\documentclass[11pt, a4paper]{article}

\newcommand{\eps}{\varepsilon}
\newcommand{\R}{\mathbb{R}}

\newcommand{\Om}{\Omega}

\newcommand{\Oes}{\Omega^\text{s}_\eps}
\newcommand{\bfu}{\textbf{u}}
\newcommand{\bfc}{\textbf{c}}
\newcommand{\bfue}{\bfu_{\eps}}

\newcommand{\bfv}{\textbf{v}}
\newcommand{\bff}{\textbf{f}}
\newcommand{\bfe}{\textbf{e}}
\newcommand{\bfchi}{\boldsymbol{\chi}}
\newcommand{\bfxi}{\boldsymbol{\xi}}
\newcommand{\bfeta}{\boldsymbol{\eta}}
\newcommand{\tx}{(t,x)}

\newcommand{\txxe}{\left(t,x,\frac{x}{\eps}\right)}
\newcommand{\bfF}{\textbf{F}}
\newcommand{\bfS}{\textbf{S}}
\newcommand{\bfSe}{\textbf{S}_\eps}
\newcommand{\bfFe}{\textbf{F}_\eps}
\newcommand{\bfD}{\textbf{D}}
\newcommand{\Ahom}{\textbf{A}^*}
\newcommand{\Jhom}{J^*}

\newcommand{\Dhom}{\textbf{D}^*}

\newcommand{\txy}{(t,x,y)}
\newcommand{\txdy}{(t,x,\cdot_y)}

\newcommand{\adjnorm}[1]{\left\lVert#1\right\rVert}
\newcommand{\norm}[1]{\|#1\|}

\newcommand{\Je}{J_\eps}
\newcommand{\ce}{c_\eps}
\newcommand{\De}{\textbf{D}_\eps}

\newcommand{\wh}[1]{\widehat{#1}}

\newcommand{\C}{\mathcal{C}}
\newcommand{\bs}[1]{\boldsymbol{#1}}
\newcommand{\hrdata}[1]{\hyperref[tab:data]{\textbf{#1}}}
\newcommand{\hrdataeps}[1]{\hyperref[tab:data]{\textbf{#1}$_\eps$}}
\newcommand{\ass}[1]{\hyperref[assumptions]{#1}}

\usepackage[utf8]{inputenc}
\usepackage{graphicx}
\usepackage{mathtools}
\usepackage{amssymb}
\usepackage{amsthm} 
\usepackage{mathabx} 
\usepackage[sc]{mathpazo}
\usepackage{geometry}
\usepackage{xcolor} 
\usepackage{tabularx}
\usepackage[colorlinks=true,linkcolor=blue,citecolor=blue]{hyperref}
\usepackage{appendix}
\usepackage{lipsum}
\usepackage{rotating}
\usepackage{adjustbox}
\usepackage{nicematrix, booktabs, siunitx}
\usepackage[labelfont=bf]{caption}
\usepackage{subcaption}
\usepackage{afterpage}
\usepackage{float}
\usepackage{verbatim}
\usepackage{fancyhdr} 
\usepackage[nottoc]{tocbibind}

\usepackage{pdflscape}
\usepackage{rotating}
\usepackage{placeins}
\usepackage{enumitem}
\usepackage{algorithm}
\usepackage{algpseudocode}

\graphicspath{{./images/}}

\allowdisplaybreaks
\newtheorem{theorem}{Theorem}[section]
\newtheorem{proposition}[theorem]{Proposition}
\newtheorem{lemma}[theorem]{Lemma}

\newtheorem{remark}[theorem]{Remark}

\title{Global well-posedness and numerical justification \\ of an effective micro-macro model \\for reactive transport in elastic perforated media}
\author{Jonas Knoch \and Markus Gahn \and Maria Neuss-Radu}
\date{\today}

\begin{document}
\maketitle
\begin{abstract}
In this paper, we investigate an effective  model for reactive transport in elastically deformable perforated media. This model was derived by formal asymptotic expansions in \cite{knoch2023}, starting from a microscopic model consisting of a linear elasticity problem on a fixed domain, i.e. in the \textit{Lagrangian} framework, and a problem for reactive transport on the current deformed domain, i.e. in the \textit{Eulerian} framework. The effective model is of micro-macro type and features strong non-linear couplings. Here, we prove global existence in time and uniqueness for the effective micro-macro model under a smallness assumption for the data of the macroscopic elasticity subproblem.  Moreover, we show numerically the convergence of microscopic solutions towards the solution of the effective model when the scale parameter $\eps>0$ becomes smaller and smaller, and also compute   the approximation error. The numerical justification of the formally derived effective  micro-macro model is particularly important, as rigorous analytical convergence proofs or error estimates are not available so far. Finally, we compare the effective micro-macro model with alternative, simpler effective descriptions of transport in elastic perforated media.
\end{abstract}
%

%
\section{Introduction}
In this paper, we are concerned with the analysis and numerical investigation of an effective micro-macro model for reactive transport in elastically deformable perforated media. This model was derived previously in  \cite{knoch2023} by  formal two-scale asymptotic expansions, starting from a microscopic model, and is motivated, for example, by biomedical applications concerning the transport of nutrients or drugs through deformable biological tissues, like pulmonary or cardiac tissue.  The microscopic description consists of a linear quasi-static elasticity system within the \textit{Lagrangian} framework on a periodically perforated reference domain representing the solid part of a porous medium and a system for reactive transport of multiple species in the current deformed solid domain, i.e. within the \textit{Eulerian} framework. In preparation of the formal upscaling, the problem was transformed into a \textit{unified Lagrangian} description by pulling back the transport problem onto the fixed, periodically perforated reference domain.   
The formally derived effective model is of micro-macro type and consists of effective (or macro) equations for the effective displacement and the effective concentrations, coupled with (micro) equations on the reference periodicity cell  which characterizes the microstructure. In addition, the two
subsystems for elasticity and transport are coupled via non-linear coefficients. Let us mention that a rigorous convergence proof of the microscopic solutions to the solution of the effective model  is not available so far. 

The aim of this paper is to prove the well-posedness of the effective model and to study its approximation properties by means of numerical simulations. More specifically, we prove global existence and uniqueness for the effective micro-macro problem. 
The main difficulty in our analysis is to establish the properties of the effective coefficients of the macroscopic transport problem which are nonlinearly coupled to the macroscopic displacement and the elasticity cell solutions.   To overcome this difficulty, we first show higher regularity of the elasticity cell solutions and of the macroscopic displacement, which are independent of the transport subsystem, based on suitable regularity of the data. Here we benefit from the fact that due to scale separation, the mixed boundary conditions that occur in the microscopic problems are now split into inhomogeneous \textit{Dirichlet} boundary conditions for the macroscopic displacement and homogeneous \textit{Neumann} boudary conditions for the cell-problems. Then, the well-posedness of the diffusion cell problems and of the effective reactive transport problem is proved based on a smallness assumption on the data of the macroscopic elasticity problem. Finally, we study the sensitivity of the smallness assumption with respect to microscopic geometry and the elastic properties of the perforated material by means of numerical simulations, and identify critical settings that allow greater flexibility in the choice of data for the macroscopic elasticity problem.
\par 
Next, we investigate numerically the convergence of the solutions to the microscopic problems towards the solution of the effective micro-macro problem for smaller and smaller values of the scale parameter $\eps > 0$. Our convergence results, yielding estimated convergence orders of approximately $1$ and $\frac{1}{2}$ in the $L^2$ and $H^1$-norm, respectively, numerically justify the formally derived effective micro-macro model. To further illustrate the approximation quality of the effective model, we present numerical simulations of the approximation errors, in particular when the first order corrector terms are added to the macroscopic solutions. Finally, we compare our effective micro-macro model with two simpler effective descriptions of transport in elastic perforated media and see, that those alternative approaches cannot capture all effects of the deformation on the transport which are observed for our effective micro-macro model. Let us mention, that the numerical investigations in this paper are using the computational framework developed in \cite{knoch2023} based on the open source finite element library \textit{deal.II} \cite{dealii2023,dealii2019}.\par  
Let us give an overview on the literature that is related to our work. Rigorous homogenization techniques for elastic composites or perforated media are extensively studied in \cite{bakhvalov1984,oleinik1992}. Effective models for transport processes in elastic porous media have been first developed 
in a purely \textit{Lagrangian} framework, i.e., for a linearized microscopic solid-fluid interface. Here, we mention \cite{jaeger2009_2,jaeger2011}, where a model for biological tissues, accounting for the deformation of the cellular structure, fluid flow in the extracellular space and transport in both phases
was introduced and homogenized, and \cite{brun2018}, where a system of linear thermoelasticity coupled to fluid flow in a porous medium was upscaled by formal asymptotic expansions. In both cases the resulting effective models consisted of an effective transport system coupled to a \textit{Biot}-type model \cite{biot1956, gilbert2000,clopeau2001,rohan2020}. 
\par
Fluid flow in a deformable porous medium in \textit{Lagrangian}/\textit{Eulerian} description was  considered, e.g., in \cite{brown2014} where a model featuring \textit{Stokes} flow in the pore space in \textit{Eulerian} description and linear elasticity in the solid matrix in the \textit{Lagrangian} description was formally homogenized by using an \textit{Arbitrary Lagrangian}/\textit{Eulerian} framework, leading to a nonlinear \textit{Biot} model on the macroscale. The model from \cite{brown2014} was generalized in \cite{collis2017} by considering a hyperelastic, morpho-poro-elastic medium as well as transport in both, the solid and the fluid. However, the explicit representation of the effective coefficients was not specified. \par
Further results related to our work deal with processes in porous media with an evolving microstructure where the evolution is prescribed \textit{a priori}, see, e.g. \cite{gahn2021,peter2007} for reaction-diffusion models or \cite{eden2017} for a model in thermoelasticity where the evolution is due to phase transitions.
In both cases, effective models are derived using the method of two-scale convergence. 
Rigorous homogenization results of problems including a microscopic free boundary are rare, see for example \cite{gahn2023} for a microstructure characterized by  spherical grains, where the radii depend on the solute concentration at the surface. 
\par
Numerical simulation frameworks for micro-macro models have been developed e.g. for problems of mineral precipitation/dissolution, see e.g. \cite{ray2019,olivares2021}. Numerical convergence of microscopic solutions to solutions of effective models has been shown e.g. in \cite{allaire2009}, where 
a non-linear and non-local heat transfer problem was considered,
 in \cite{khoa2021}
 for a semi-linear elliptic problem, and in \cite{ye2021} 
 for a \textit{Steklov} eigenvalue problem. 
 \par

The paper is structured as follows. In Section \ref{sec:the_models} both the microscopic model in unified \textit{Lagrangian} description as well as the 
effective micro-macro model from \cite{knoch2023} are introduced. The proof for the global existence and uniqueness for the effective micro-macro model is given in Section \ref{sec:analysis}, together with the numerical characterization of the smallness condition, see Section \ref{subseq:characterization}. In Section \ref{sec:numerical_justification}  we show numerical convergence of the microscopic problems towards the micro-macro model and also investigate the approximation error. The paper wraps up with a conclusion and an outlook in Section \ref{sec:discussion}.
\section{The microscopic model and the effective micro-macro model}
\label{sec:the_models}
\begin{figure}[t]
\centering
\includegraphics[width=0.6\textwidth]{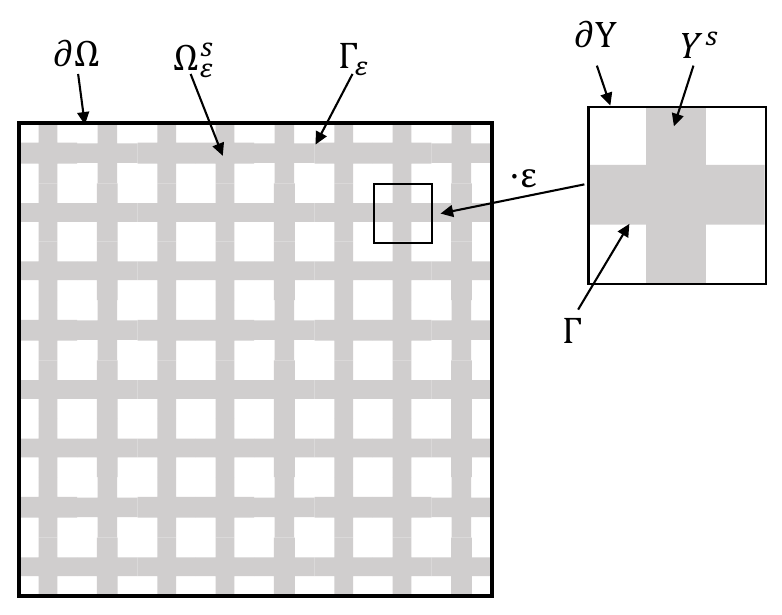}
\caption{Sketch of the microscopic geometry of the reference domain, see \cite{knoch2023}.}
\label{fig:micro_domain}
\end{figure}

We start by defining the geometry of the reference domains for the microscopic elasticity-transport problem as well as for the macroscopic (effective) micro-macro problem. Let $\omega \subset \R^n$ ($n = 2,3$ for the physically relevant cases) be an unbounded domain with 1-periodic structure, i.e. $\omega + \textbf{k} = \omega$ for all $\textbf{k} \in \mathbb{Z}^n$. The reference (periodicity) cell is defined by $Y^s := \omega \cap (0,1)^n$ and the inner part of its boundary is denoted by $\Gamma := \text{int}(\partial Y^s \setminus \partial (0,1)^n)$. Let $\eps > 0$ be the scale parameter. The microscopic perforated reference domain $\Oes$ is given by
$$
\Oes := \Om \cap \eps \omega,
$$
and represents the solid part of a deformable porous medium with reference domain $\Om \subset \R^n$.
The internal (microscopic) boundary of the perforated domain $\Oes$
is denoted by $\Gamma_\eps := \text{int}(\partial \Oes \setminus \partial \Om)$. See Figure \ref{fig:micro_domain} for a sketch of the microscopic reference domain.
\par 
The microscopic model in \textit{Lagrangian} formulation reads as follows: For $T>0$, find the displacement $\bfu_\eps \colon (0,T) \times \Oes \rightarrow \R^n$ and concentrations $c_\eps^m \colon (0,T) \times \Oes \rightarrow \R$, for $m=1,...,N_c$, with $\bfc_\eps := (c_\eps^1, ..., c_\eps^{N_c})^T$, such that
\begin{subequations}
\label{eq:micro_model_on_reference_domain}
\begin{align}
 - \nabla \cdot \left( \textbf{A} \bfe(\bfue) \right) &= \bff_e & \mbox{ in } & (0,T) \times \Oes, \label{eq:micro_model_on_reference_domain:a} \\
- \textbf{A}\bfe(\bfue) \cdot \textbf{n} &= 0 & \mbox{ on } & (0,T) \times \Gamma_\eps, \label{eq:micro_model_on_reference_domain:b} \\
\bfue & = \bfu_D & \mbox{ on } & (0,T) \times \partial \Oes \setminus \Gamma_\eps, \label{eq:micro_model_on_reference_domain:c} \\
\partial_t \left( \Je \ce^m \right) - \nabla \cdot \left( \De \nabla \ce^m \right) &= \Je f_d^m(\textbf{c}_\eps) & \mbox{ in } & (0,T) \times \Oes, \label{eq:micro_model_on_reference_domain:d} \\
- \De \nabla \ce^m \cdot \textbf{n} &= 0 & \mbox{ on } & (0,T) \times \Gamma_\eps, \label{eq:micro_model_on_reference_domain:e} \\
\ce^m &= 0 & \mbox{ on } & (0,T) \times \partial \Oes \setminus \Gamma_\eps ,\label{eq:micro_model_on_reference_domain:f} \\
\ce(0,\cdot) &=  c^{0,m} & \mbox{ in } & \Oes. \label{eq:micro_model_on_reference_domain:g}
\end{align}
\end{subequations}%
Here, $\textbf{A} \in \R^{n\times n \times n \times n}$ is a constant fourth-order elasticity tensor and $\bff_e \colon (0,T) \times \Oes \rightarrow \R^n$ is a body force acting on the solid phase. $\bfe(\textbf{w}) := \frac{1}{2}\left( \nabla \textbf{w} + (\nabla \textbf{w})^T\right)$ denotes the symmetric gradient and on the outer part of the boundary of $\Oes$, the displacement is prescribed by a function $\bfu_D\colon (0,T) \times \partial \Oes \setminus \Gamma_\eps \rightarrow \R^n$ with $\bfu_D(0,\cdot) = 0$. $\textbf{n}$ denotes the outer normal to the microscopic boundary $\Gamma_\eps$. \par 
The deformation $\bfS_\eps \colon (0,T) \times \Oes \rightarrow \R^n$ of the microscopic domain is given by 
\begin{equation*}
\bfS_\eps (t,x) := x + \bfue(t,x).
\end{equation*}
In \cite{knoch2023}, the deformation $\bfS_\eps$ was used to pull back the transport problem formulated on the current deformed domain 
$$
\Oes(t) := \{ \wh{x} \in \R^n \mid \wh{x} = \bfSe(t,x), \, x \in \Oes \},
$$
i.e. within the \textit{Eulerian} framework, onto the fixed reference domain $\Oes$, leading to the subsystem \eqref{eq:micro_model_on_reference_domain:d}-\eqref{eq:micro_model_on_reference_domain:g} in the \textit{Lagrangian} framework. This transformation introduces the time and space dependent coefficients $J_\eps$ and $\bfD_\eps$ for the new transport problem on the fixed reference domain which carry the information about the deformation of the domain. The coefficients $J_\eps$ and $\bfD_\eps$ are nonlinear functions of the deformation gradient 
$
\bfFe(t,x) := \nabla \bfSe(t,x),
$
and are given by 
$$
J_\eps(t,x) := \det(\bfFe(t,x)), \quad 
\bfD_\eps(t,x) := [J_\eps \bfFe^{-1} \wh{\bfD} \bfFe^{-T}](t,x).
$$
$\wh{\bfD} \in \R^{n\times n}$ denotes the constant diffusion tensor and $f_d^m \colon \R^n \rightarrow \R$ denotes the reaction term of the $m$-th species, $m = 1,...,N_c$. Moreover, the initial concentrations are given by the functions $c^{0,m} \colon \Om \rightarrow \R$, restricted to the microscopic domain, for all species $m=1,...,N_c$.
\par \bigskip
In this \textit{Lagrangian} formulation, the microscopic problem \eqref{eq:micro_model_on_reference_domain:a}-\eqref{eq:micro_model_on_reference_domain:g} is accessible for the method of two-scale asymptotic expansion \cite{bensoussan2011,sanchez-palencia1980,bakhvalov1984} in order to formally derive an \textit{effective} or \textit{homogenized} model, which does not  explicitly feature the microstructure anymore and is therefore easier to handle from a numerical point of view. The basic principle of this method is to postulate representations of the unknowns, $\bfue$ and $\bfc_\eps$, as power series in terms of the scale parameter $\eps$, i.e.
\begin{equation} \label{eq:two_scale_expansions}
\bfue\tx = \sum_{i=0}^\infty \eps^i \bfu_i \left(t,x,\frac{x}{\eps}\right), \quad \bfc_\eps\tx = \sum_{i=0}^\infty \eps^i \bfc_i \left(t,x,\frac{x}{\eps}\right),
\end{equation}
with coefficient functions $\bfu_i(t,x,y)$ and $\textbf{c}_i(t,x,y)$ depending on the macroscopic variable $x \in \Om$ and the microscopic variable $y \in Y^s$ and being $Y^s$-periodic with respect to $y$. For a detailed derivation of the effective model see \cite{knoch2023}.
\par \bigskip
The effective micro-macro model, derived from the microscopic problem \eqref{eq:micro_model_on_reference_domain:a}-\eqref{eq:micro_model_on_reference_domain:g}, reads as follows: Find $\bfu\colon (0,T) \times \Omega \rightarrow \R^n$ and $c^m\colon (0,T) \times \Omega \rightarrow \R$, for $m=1,...,N_c$, with $\bfc := (c^1,..., c^{N_c})^T$, such that 
\begin{subequations}
\label{eq:effective_micro_macro_model}
\begin{align}
- \nabla \cdot \left( \Ahom \bfe( \bfu ) \right) &= |Y^s|\bff_\text{e} & \mbox{ in } & (0,T) \times \Omega, \label{eq:effective_micro_macro_model:a}
\\
\bfu  &= \bfu_D &\mbox{ on } & (0,T) \times \partial \Omega, \label{eq:effective_micro_macro_model:b}\\
\partial_t \left( \Jhom c^m \right) - \nabla \cdot \left( \Dhom\nabla c^m \right)  &=  \Jhom f_d^m(\bfc) & \mbox{ in } & (0,T) \times \Omega, \label{eq:effective_micro_macro_model:c}\\
c^m  & = 0 & \mbox{ on } & (0,T) \times \partial \Omega, \label{eq:effective_micro_macro_model:d}\\
c^m(0,\cdot)  & = c^{m,0} & \mbox{ in } & \Omega. \label{eq:effective_micro_macro_model:e}
\end{align}
\end{subequations}
Here, $\Ahom$, $\Jhom$ and $\Dhom$ are homogenized coefficients defined as follows
\begin{align}
A^*_{ijkl} &= \int_{Y^s} \textbf{A}\left(\bfe_y(\boldsymbol{\chi}_{kl}(y)) + \textbf{M}_{kl}\right):\left(\bfe_y(\boldsymbol{\chi}_{ij}(y)) + \textbf{M}_{ij}\right)\textrm{d}y, \label{eq:Ahom}\\
\Jhom(t,x) &= \int_{Y^s} J_0(t,x,y) \,\textrm{d}y, \label{eq:Jhom} \\
D^*_{ij}(t,x) &= \int_{Y^s} \textbf{D}_0(t,x,y)\left(\mathbf{e}_j + \nabla_y\eta_j(t,x,y)\right) \cdot \left(\mathbf{e}_i + \nabla_y\eta_i(t,x,y)\right) \textrm{d}y, \label{eq:Dhom}
\end{align}
via the solutions $\bfchi_{ij}\colon Y^s \rightarrow \R^n$, of the elasticity cell problems 
\begin{subequations}
\label{eq:elasticity_cell_problems}
\begin{align}
- \nabla_y \cdot \left[\mathbf{A}\left( \textbf{M}_{ij} + \bfe_y(\boldsymbol{\chi}_{ij}(y)) \right)\right] &= 0 &\mbox{ in }& Y^s,
\\
- \mathbf{A}\left( \textbf{M}_{ij} + \bfe_y(\boldsymbol{\chi}_{ij}(y)) \right) \cdot \mathbf{n}_{\Gamma} &= 0 &\mbox{ on }& \Gamma,\\
\boldsymbol{\chi}_{ij} \mbox{ is } Y^s \mbox{-periodic in $y$, } \, \int_{Y^s} \boldsymbol{\chi}_{ij}(y) \textrm{d}y &= 0, 
\end{align}
\end{subequations}
and the solutions $\eta_i \colon(0,T) \times \Omega \times Y^s \rightarrow \R$ of the diffusion cell problems
\begin{subequations}
\label{eq:diffusion_cell_problems}
\begin{align}
-\nabla_y \cdot  \left[\mathbf{D}_0\txy(\mathbf{e}_i + \nabla_y \eta_i\txy)\right] &= 0 &\mbox{ in }& (0,T) \times \Omega \times Y^s, \\
- \mathbf{D}_0\txy\left(\mathbf{e}_i + \nabla_y \eta_i\txy\right) \cdot \mathbf{n}_\Gamma &= 0 &\mbox{ on } & (0,T) \times \Omega \times \Gamma, \\
\eta_i \text{ is $Y^s$-periodic in } y, \quad \int_{Y^s} \eta_i\txy \,\textrm{d}y &= 0,
\end{align}
\end{subequations}%
for $i,j,k,l = 1, ...,n$. Here, $\mathbf{n}_\Gamma$ denotes the outer normal to the cell boundary $\Gamma$. Furthermore, $J_0$ and $\mathbf{D}_0$ are given by
\begin{align}
J_0(t,x,y) &= \det (\mathbf{F}_0\txy),\label{def:J_0}\\
\mathbf{D}_0\txy &= [J_0 \mathbf{F}_0^{-1}\widehat{\mathbf{D}}\mathbf{F}^{-T}_0]\txy,\label{def:D_0}
\end{align}
where
\begin{equation}\label{def:F_0}
\mathbf{F}_0(t,x,y) = \mathbf{E}_n + \nabla_x \bfu\tx + \sum_{i,j = 1}^n \bfe_x(\bfu)_{ij}\tx \nabla_y \boldsymbol{\chi}_{ij}(y).
\end{equation}
Here, $\bfu = \bfu(t,x)$ and $\bfc = \bfc(t,x)$ denote the zeroth-order terms of the expansions \eqref{eq:two_scale_expansions} which are in fact independent of the microscopic variable. The first-order terms in \eqref{eq:two_scale_expansions} are given in terms of the zero-order terms and the cell solutions via
\begin{equation}\label{eq:correctors}
\eps \bfu_1\txxe = \eps \sum_{i,j=1}^n \bfe(\bfu)_{ij}(t,x) \bfchi_{ij}\left(\frac{x}{\eps}\right), \quad \eps c_1^m\txxe = \eps \sum_{i=1}^n \partial_{x_i}c^m(t,x) \eta_i\txxe, 
\end{equation}
for $m = 1,...,N_c$. \par
In the definition of the effective elasticity tensor $\Ahom$, see \eqref{eq:Ahom}, and in the formulation of the elasticity cell problems, see \eqref{eq:elasticity_cell_problems}, we have used the notation $\textbf{M}_{ij} := \frac{1}{2}(\bfe_i \otimes \bfe_j + \bfe_j \otimes \bfe_i) \in \mathbb{S}^n$, $i,j=1,...,n$, where $\bfe_i$, denotes the $i$-th canonical basis vector of $\R^n$ and $\otimes \colon \R^n \times \R^n \rightarrow \R^{n \times n}$ is the dyadic product. Here $\mathbb{S}^n := \{ \textbf{B} \in \R^{n\times n} \mid \textbf{B} = \textbf{B}^T \}$ denotes the space of symmetric matrices. 
\section{Global existence in time and uniqueness for the effective micro-macro model}
\label{sec:analysis}
\paragraph{Notation}
We will use the following notations for function spaces of periodic functions. $\C^\infty_\text{per}(\widebar{\omega})$ is the space of infinitely differentiable functions on $\omega$ whose derivatives can be continuously extended to the boundary and which are $1$-periodic in $y_1,...,y_n$. $H^1_\text{per}(\omega)$ is the closure of $\C^\infty_\text{per}(\widebar{\omega})$ with respect to the norm $\norm{.}_{H^1(Y^s)}$. Additionally, we introduce $W_\text{per}(\omega) := \{ v \in H^1_\text{per}(\omega) \mid \int_{Y^s} v \textrm{d}y = 0 \}$. Moreover, we write short $\Om_T := (0,T) \times \Om$. For two vectors $\bs{\xi} = (\xi_1,...,\xi_n)^T,\bs{\eta} = (\eta_1,...,\eta_n)^T \in \R^n$ and two matrices $\textbf{A} = (A_{ij})_{i,j=1,...,n}, \textbf{B} = (B_{ij})_{i,j=1,...,n}\in \R^{n\times n}$, we define the scalar products
$\bs{\xi} \cdot \bs{\eta} := \sum_{i=1}^n \xi_i  \eta_i$  and $\textbf{A}:\textbf{B} := \sum_{i,j=1}^n A_{ij}B_{ij}.$ In addition, we write for the corresponding norms $|\bs{\xi}|:= (\bs{\xi}\cdot \bs{\xi})^{\frac{1}{2}}$ and  $|\textbf{A}|:= (\textbf{A}:\textbf{A})^{\frac{1}{2}}$. $\textbf{E}_n$ is the identity matrix of size $n\times n$. Furthermore, we use the norms $\norm{\textbf{A}}_2 := \sup_{|\bs{\xi}|= 1} |\textbf{A}\bs{\xi}|$ and $\norm{\textbf{A}}_\infty := \max_{i,j=1,...,n} |A_{ij}|$.
\paragraph{Assumptions on the data}
\label{assumptions}
\begin{itemize}
\item[(A1)] The boundary $\partial \Omega$ of the macroscopic domain $\Omega$ is of class $\mathcal{C}^{3}$.
\item[(A2)] The boundary $\partial \omega$ of the unbounded $1$-periodic domain $\omega$ is of class $\mathcal{C}^\infty$ and the boundary of the reference cell $\partial Y^s$  is of class $\mathcal{C}^{0,1}$. 
\item[(A3)] The elasticity tensor $\textbf{A}$ for the microscopic problem is a constant fourth order tensor $\textbf{A} \in \mathcal{L}(\mathbb{S}^n,\mathbb{S}^n)$ which fulfills the symmetry properties
\begin{equation}
\label{eq:symmetry_of_A}
A_{ijkl} = A_{jikl} = A_{klij} \quad \text{ for } i,j,k,l = 1,...,n, 
\end{equation}
and the ellipticity condition on the space of symmetric matrices
$$
\mathbf{A}\mathbf{B}:\mathbf{B} \geq \alpha |\textbf{B}|^2 \quad \text{ for all } \mathbf{B} \in \mathbb{S}^n \text{ with } \alpha > 0. 
$$
\item[(A4)] The diffusion tensor $\widehat{\mathbf{D}}$ for the microscopic problem is a constant second order tensor $\widehat{\textbf{D}} \in \mathcal{L}(\R^n, \R^n)$ which is symmetric, i.e.
\begin{equation} \label{eq:symmetry_of_D_hat}
\widehat{D}_{ij} = \widehat{D}_{ji} \quad \text{ for } i,j = 1,...,n,
\end{equation}
and fulfills the ellipticity condition
\begin{equation}\label{eq:ellipticity_of_D_hat}
\widehat{\mathbf{D}}\boldsymbol{\xi}\cdot \boldsymbol{\xi} \geq \delta |\boldsymbol{\xi}|^2 \quad \text{ for all } \boldsymbol{\xi} \in \R^n \text{ with } \delta > 0,
\end{equation}
\item[(A5)] Let $\sigma \in (0,1)$. The body force density $\bff_e$ is given by $\bff_e\tx = \nabla_x \cdot \textbf{F}_e\tx$ with $\textbf{F}_e \in \mathcal{C}^1([0,T];\mathcal{C}^{1,\sigma}(\widebar{\Omega})^{n\times n})$. The \textit{Dirichlet} boundary function satisfies $\bfu_D \in \mathcal{C}^1([0,T];\mathcal{C}^{2,\sigma}(\widebar{\Omega})^{n})$.
\item[(A6)] The reaction term $\bff_d := (f_d^1, ..., f_d^{N_c})^T$ of the macroscopic diffusion problem satisfies $\bff_d \in \mathcal{C}^{0,1}(\R^{N_c}; \R^{N_c})$, and the initial value $\bfc_0 := (c_0^1, ..., c_0^{N_c})^T$ satisfies $\bfc_0 \in H^1_0(\Omega)^{N_c}$. 
\end{itemize}
%
%
Let us now state the main results of this section.
\begin{theorem}\label{theorem:main_theorem_1}
Suppose the assumptions (\ass{A1})-(\ass{A3}) and (\ass{A5}) are satisfied. Then, there exists a unique solution 
$$
(\bfchi_{ij}, \bfu) \in  \mathcal{C}^\infty_{per}(\widebar{\omega})^n \times \mathcal{C}^1([0,T], \mathcal{C}^{2,\sigma}(\widebar{\Om})^n), \quad i,j=1,...,n,
$$ of the elasticity subsystem \eqref{eq:effective_micro_macro_model:a}-\eqref{eq:effective_micro_macro_model:b}, \eqref{eq:elasticity_cell_problems} and we have
\begin{equation}  \label{eq:u_full_estimate}
\norm{\bfu}_{\mathcal{C}^1([0,T];\,\mathcal{C}^{2,\sigma}(\widebar{\Omega})^n)} \leq C_\bfu \left( \norm{\textbf{F}_e}_{\mathcal{C}^1([0,T];\,\mathcal{C}^{1,\sigma}(\widebar{\Omega})^{n\times n})} + \norm{\bfu_D}_{\mathcal{C}^1([0,T];\,\mathcal{C}^{2,\sigma}(\widebar{\Omega})^n)} \right)
\end{equation}
with $C_\bfu = C_\bfu(\Omega, \sigma,\Ahom) > 0$.
\end{theorem}
\begin{theorem}\label{theorem:main_theorem_2}
Suppose assumptions (\ass{A1})-(\ass{A6}) are satisfied and and let
\begin{equation}\label{eq:smallness_assumption}
\norm{\textbf{F}_e(t)}_{\mathcal{C}^{1,\sigma}(\widebar{\Om})^{n \times n}} + \norm{\bfu_D(t)}_{\mathcal{C}^{2,\sigma}(\widebar{\Om})^n} < \frac{1}{(n+n^3C_{\bfchi})C_\bfu}, \quad \text{for all } t \in [0,T],
\end{equation}
with $C_\bfu$ from \eqref{eq:u_full_estimate} and
\begin{equation}\label{eq:def_c_chi}
   C_{\bfchi} := \sup_{y \in \widebar{Y^s}, i,j = 1,...,n} \norm{\nabla_y \bfchi_{ij}(y)}_\infty.
\end{equation}
Then, for $i=1,...,n$, there exists a unique solution
$$
(\eta_i, \bfc) \in \mathcal{C}^1([0,T]; \C^{1,\sigma}(\widebar{\Om};W_{per}(\omega))) \times L^2(0,T; H^2(\Om)^{N_c})\cap L^\infty(0,T;H^1_0(\Om)^{N_c}),
$$
of the transport subsystem with $\partial_t\bfc \in L^2(0,T;L^2(\Om)^{N_c})$, $\bfc(0) = \bfc^{0}$, which satisfy, for $i=1,...,n$, and for all $\zeta \in W_\text{per} (\omega)$ the variational problem in $Y^s$
\begin{equation} \label{eq:variational_dcps}
\int_{Y^s} \bfD_0(t,x,y) \nabla_y \eta_i(t,x,y) \cdot \nabla_y \zeta(y) \textrm{d}y = -\int_{Y^s} \bfD_0(t,x,y) \bfe_i \cdot \nabla_y \zeta(y) \textrm{d}y
\end{equation}
for all $(t,x) \in \widebar{\Om}_T$ and for $m=1,...,N_c$ and for all $\varphi \in H^1_0(\Om)$ the variational problem in $\Om$
\begin{equation}\label{eq:variational_mdp}
\int_{\Omega} \partial_t c^m \varphi \textrm{d}x + \int_\Om \left( \frac{1}{\Jhom}\Dhom \nabla c^m \cdot \nabla \varphi - \frac{1}{{\Jhom}^2}{\Dhom}^T \nabla \Jhom \cdot \nabla c^m \varphi+ \frac{\partial_t \Jhom}{\Jhom}c^m \varphi \right) \textrm{d}x = \int_\Om  f_d^m(\bfc) \varphi \textrm{d}x
\end{equation}
for all $t \in (0,T)$.
\end{theorem}
\subsection{Regularity of the displacement}
We start by summarizing known results about the elasticity cell solutions and the effective elasticity tensor.
\begin{proposition}
\label{prop:elasticity_cell_solutions_and_Ahom}
Suppose assumptions (\ass{A2}) and (\ass{A3}) are satisfied. Then, for $i,j=1,...,n$, there exists a unique solution $\bfchi_{ij} \in \mathcal{C}^\infty_{per}(\widebar{\omega})^n$ to the elasticity cell problem \eqref{eq:elasticity_cell_problems}.\\
The effective elasticity tensor $\Ahom$ defined in \eqref{eq:Ahom} fulfills the same symmetry properties as required for $\textbf{A}$ in \eqref{eq:symmetry_of_A}, i.e.
\begin{equation}
\label{eq:Ahom_symmetry}
A^*_{ijkl} = A^*_{klij} = A^*_{jikl} \quad \text{ for } i,j,k,l = 1,...,n,
\end{equation}
and is elliptic on the space of symmetric matrices, i.e.
\begin{equation}
\label{eq:Ahom_coercivity}
\Ahom \textbf{B}:\textbf{B} \geq \alpha |\textbf{B}|^2 \quad \text{ for all } \textbf{B} \in \mathbb{S}^n \text{ with } \alpha > 0 \text{ from (A3)}.
\end{equation}
\end{proposition}
\begin{proof}
The result concerning the elasticity cell solutions is given, e.g., in \cite[Chapter~I, Theorem~6.2]{oleinik1992}. The properties of $\Ahom$ are shown, e.g., in \cite[Chapter~II, Theorem~1.1]{oleinik1992}.
\end{proof}
To show higher regularity results for the effective displacement $\bfu$, we first remark that the homogenized elasticity tensor $\Ahom$ is in general anisotropic. This has been previously pointed out in the homogenization literature, see e.g. \cite[Chapter 4.4]{bakhvalov1984}, and is also observed in numerical simulations, see e.g. \cite{knoch2023}. Thus, higher regularity results for solutions to linear elasticity problems, as they are given, e.g. in \cite{ciarlet1988}, where isotropic media were considered, are not applicable to our model. Instead, our proof is based on regularity results for more general elliptic systems given in \cite{giaquinta2013}.
\begin{proof}[Proof of Theorem \ref{theorem:main_theorem_1}]
We first show that the homogenized elasticity tensor $\Ahom$ fulfills the \textit{Legendre}-\textit{Hadamard} condition, i.e. $\exists \lambda > 0$ such that
\begin{equation}\label{eq:legendre_hadamard_condition}
\sum_{i,j,k,l=1}^n A^*_{ijkl}\eta_k\xi_l\eta_i\xi_j \geq \lambda |\boldsymbol{\eta}|^2 |\boldsymbol{\xi}|^2, \quad \forall \boldsymbol{\eta}, \boldsymbol{\xi} \in \R^n.
\end{equation}
With $\textbf{B} := \bfeta \otimes \bfxi$, and $\textbf{B}^S, \textbf{B}^A$ the symmetric and anti-symmetric part of $\textbf{B}$, we have
\begin{align*}
\sum_{i,j,k,l=1}^n A^*_{ijkl}\eta_k\xi_l\eta_i\xi_j &= \Ahom \textbf{B}:\textbf{B}  \stackrel{\eqref{eq:Ahom_symmetry}}{=} \Ahom\textbf{B}^S:\textbf{B}^S \\ & \stackrel{\eqref{eq:Ahom_coercivity}}{\geq} \alpha |\textbf{B}^S|^2 = \alpha \left( |\textbf{B}|^2 - \textbf{B}:\textbf{B}^A \right)\\
&= \frac{\alpha}{2} |\textbf{B}|^2 + \frac{\alpha}{2} \textbf{B} : \textbf{B}^T.
\end{align*}
Noting that $|\textbf{B}|^2 = |\bfeta|^2 |\bfxi|^2$ and 
$$
\textbf{B} : \textbf{B}^T = \sum_{i,j}^n \eta_i \xi_j \eta_j \xi_i = \left( \sum_{i=1}^n \eta_i \xi_i \right)^2 \geq 0
$$
shows that $\Ahom$ fulfills \eqref{eq:legendre_hadamard_condition} with $\lambda = \frac{\alpha}{2}$.
Now, from \cite[Section 5.4.5, Theorem 5.21]{giaquinta2013}, it follows that there exists a unique solution $\bfu(t) \in \C^{2,\sigma}(\widebar{\Om})^n$ with
\begin{equation} \label{eq:giaquinta_estimate}
\norm{\bfu(t)}_{\mathcal{C}^{2,\sigma}(\widebar{\Omega})^n} \leq C_{\bfu} \left( \norm{\textbf{F}_e(t)}_{\mathcal{C}^{1,\sigma}(\widebar{\Omega})^{n\times n}} + \norm{\bfu_D(t)}_{\mathcal{C}^{2,\sigma}(\widebar{\Omega})^n} \right) \quad \forall t \in [0,T],
\end{equation}
where $C_{\bfu} = C_\bfu(\Omega, \sigma, \Ahom) > 0.$ %
Finally we show that $\bfu \in \mathcal{C}^1([0,T]; \mathcal{C}^{2,\sigma}(\widebar{\Om})^n)$. To this end, let us consider the following problems for all $t \in (0,T)$:
\begin{equation} \label{eq:pb_for_u}
\begin{aligned}
-\nabla \cdot \left( \Ahom \bfe(\bfu(t)) \right) &= |Y^s|\nabla \cdot \bfF_e(t) & \mbox{ in } & \Om, \\
\bfu(t) &= \bfu_D(t) & \mbox{ on } &  \partial \Om,
\end{aligned}
\end{equation}
\begin{equation}\label{eq:pb_for_u+h}
\begin{aligned}
-\nabla \cdot \left( \Ahom \bfe(\bfu(t+h)) \right) &= |Y^s|\nabla \cdot \bfF_e(t+h) & \mbox{ in } &  \Om,\\
\bfu(t+h) &= \bfu_D(t+h) & \mbox{ on } &  \partial \Om,
\end{aligned}
\end{equation}
for $h>0$ small enough and 
\begin{equation}\label{eq:pb_for_v}
\begin{aligned}
-\nabla \cdot \left( \Ahom \bfe(\bfv(t)) \right) &= |Y^s|\nabla \cdot \bfF'_e(t) & \mbox{ in } &  \Om,\\
\bfv(t) &= \bfu'_D(t) & \mbox{ on } &  \partial \Om.
\end{aligned}
\end{equation}
To obtain the continuity of $\bfu$ with respect to time, let us subtract \eqref{eq:pb_for_u} from \eqref{eq:pb_for_u+h} to get the following problem:
\begin{equation*}
\begin{aligned}
- \nabla \cdot \left( \Ahom \bfe( \bfu(t+h)-\bfu(t) ) \right) &= |Y^s| \nabla \cdot ( \bfF_e(t+h) - \bfF_e(t) ) & \mbox{ in } & \Om,\\
\bfu(t+h) - \bfu(t) & = \bfu_D(t+h) - \bfu_D(t) & \mbox{ on } & \partial \Om
\end{aligned}
\end{equation*}
for all $t \in (0,T)$. With the estimate \eqref{eq:giaquinta_estimate} from the first part of the proof, we infer for all $t \in (0,T)$
\begin{align*}
\lVert \bfu(t+h) & - \bfu(t) \rVert_{\mathcal{C}^{2,\sigma}(\widebar{\Om})^n} \\
& \leq C_\bfu \left( \adjnorm{ \bfF_e(t+h) - \bfF_e(t) }_{\mathcal{C}^{1,\sigma}(\widebar{\Om})^{n \times n}} + \adjnorm{ \bfu_D(t+h) - \bfu_D(t)}_{\mathcal{C}^{2,\sigma}(\widebar{\Om})^n} \right) \\
&  \xrightarrow{h \rightarrow 0} 0,
\end{align*}
due to the continuity of $\bfF_e$ and $\bfu_D$ with respect to time. For the differentiability of $\bfu$, let us subtract \eqref{eq:pb_for_u} from \eqref{eq:pb_for_u+h}, divide by $h>0$ and subtract \eqref{eq:pb_for_v} to obtain, for all $t \in (0,T)$:
\begin{equation*}
\begin{aligned}
- \nabla \cdot \left( \Ahom \bfe \left( \frac{\bfu(t+h)-\bfu(t)}{h} - \bfv(t) \right) \right) &= |Y^s| \nabla \cdot \left( \frac{\bfF_e(t+h) - \bfF_e(t)}{h} - \bfF'_e(t) \right) & \mbox{ in } &  \Om,\\
\frac{\bfu(t+h) - \bfu(t)}{h} - \bfv(t) & = \frac{\bfu_D(t+h) - \bfu_D(t)}{h} - \bfu'_D(t) & \mbox{ on } &  \partial \Om.
\end{aligned}
\end{equation*}
By virtue of \eqref{eq:giaquinta_estimate} we have for all $t \in (0,T)$ that
\begin{align*}
\Bigg \lVert &\frac{\bfu(t+h)-\bfu(t)}{h} - \bfv(t) \Bigg\rVert_{\mathcal{C}^{2,\sigma}(\widebar{\Om})^n} \\
& \leq C_\bfu \left( \adjnorm{ \frac{\bfF_e(t+h) - \bfF_e(t)}{h} - \bfF'_e(t) }_{\mathcal{C}^{1,\sigma}(\widebar{\Om})^{n \times n}} + \adjnorm{ \frac{\bfu_D(t+h) - \bfu_D(t)}{h} - \bfu'_D(t)}_{\mathcal{C}^{2,\sigma}(\widebar{\Om})^n} \right) \\
&  \xrightarrow{h \rightarrow 0} 0,
\end{align*}
due to the differentiability of $\bfF_e$ and $\bfu_D$ with respect to time. Hence we conclude that $\bfu$ is differentiable in time with $\partial_t \bfu = \bfv$. The continuity of $\bfv$ follows with similar steps as before from the continuity of $\partial_t \bfF_e$ and $\partial_t \bfu_D$. Also, we can extend $\bfu$ and $\bfv$ continuously to the boundary of the time interval based on the regularity properties of $\bfF_e$ and $\bfu_D$. Therefore we have $\bfu \in \mathcal{C}^1([0,T], \mathcal{C}^{2,\sigma}(\widebar{\Om})^n)$. Eventually, we obtain with estimate \eqref{eq:giaquinta_estimate}, applied to the problems \eqref{eq:pb_for_u} and \eqref{eq:pb_for_v}, that
\begin{align*}
\sup_{t \in [0,T]} & \adjnorm{\bfu(t)}_{\mathcal{C}^{2,\sigma}(\widebar{\Om})^n} + \sup_{t \in [0,T]} \adjnorm{\bfv(t)}_{\mathcal{C}^{2,\sigma}(\widebar{\Om})^n} \\
& \leq C_\bfu \Big( \sup_{t \in [0,T]} \adjnorm{\bfF_e(t)}_{\mathcal{C}^{1,\sigma}(\widebar{\Om})^{n \times n}} +   \sup_{t \in [0,T]} \adjnorm{\partial_t \bfF_e(t)}_{\mathcal{C}^{1,\sigma}(\widebar{\Om})^{n \times n}} \\
& \qquad \qquad \qquad \qquad \qquad \qquad +  \sup_{t \in [0,T]} \adjnorm{\bfu_D(t)}_{\mathcal{C}^{2,\sigma}(\widebar{\Om})^{n}} +   \sup_{t \in [0,T]} \adjnorm{\partial_t \bfu_D(t)}_{\mathcal{C}^{2,\sigma}(\widebar{\Om})^{n}} \Big),
\end{align*}
which yields \eqref{eq:u_full_estimate}. 
\end{proof}
\subsection{Existence and uniqueness for the effective diffusion problem and the associated cell problems}
The main difficulty in the study of the transport subsystem is to establish the properties of the effective coefficients $\Jhom$ and $\Dhom$, see \eqref{eq:Jhom}-\eqref{eq:Dhom}, among which e.g. the positivity of $\Jhom$ is crucial for the well-posedness of the effective transport equation. To overcome this difficulty, a smallness condition on the data $\bfF_e$ and $\bfu_D$ of the macroscopic elasticity problem is assumed. \par
We start by analyzing the coefficients $J_0$ and $\bfD_0$, see \eqref{def:J_0}-\eqref{def:D_0}, which enter the definition of the effective coefficients and depend nonlinearly on the macroscopic displacement $\bfu$ and the elasticity cell problems $\bfchi_{ij}$.
\begin{proposition} \label{proposition:properties_of_J0_D0} 
Suppose assumptions (\ass{A1})-(\ass{A5}) are satisfied. Let $C_\bfu > 0$ be given as in \eqref{eq:u_full_estimate}, let $C_{\bfchi}>0$ be given as in \eqref{eq:def_c_chi}
and suppose that
\begin{equation}
\norm{\textbf{F}_e(t)}_{\mathcal{C}^{1,\sigma}(\widebar{\Om})^{n \times n}} + \norm{\bfu_D(t)}_{\mathcal{C}^{2,\sigma}(\widebar{\Om})^n} < \frac{1}{(n + n^3C_{\bfchi})C_\bfu} 
\end{equation}
for all $t \in [0,T]$. Then
\begin{enumerate}
\item $J_0 \in \mathcal{C}^1([0,T]; \C^{1,\sigma}(\widebar{\Om}; \mathcal{C}^\infty_\text{per} (\widebar{\omega})))$ with $0 < J_0\txy \leq \Gamma$, where $\Gamma = n^{\frac{n}{2}}\left(1 + \frac{1}{n}\right)^n$, for all $\txy \in \widebar{\Om}_T \times \widebar{Y}^s$.
\item $\textbf{D}_0 \in  \mathcal{C}^1([0,T]; \C^{1,\sigma}(\widebar{\Om}; \mathcal{C}^\infty_\text{per} (\widebar{\omega})^{n\times n}))$ is well-defined, symmetric and elliptic.
\end{enumerate}
\end{proposition}
\begin{proof}
\begin{enumerate}
\item Remember that according to \eqref{def:J_0} we have
$$
J_0\txy =  \det\left(\textbf{E}_n + \nabla_x \bfu(t,x) + \sum_{i,j=1}^n \bfe(\bfu)_{ij}(t,x) \nabla_y \bfchi_{ij}(y)\right).
$$
Due to the regularity of $\bfchi_{ij}$ and $\bfu$ established in Theorem \ref{theorem:main_theorem_1}, we have 
$$
\nabla_y \bfchi_{ij} \in \mathcal{C}^\infty(\widebar{\omega})^{n\times n} \text{ for $i,j=1,...,n$ and } \nabla_x \bfu \in \mathcal{C}^1([0,T], \mathcal{C}^{1,\sigma}(\widebar{\Om})^{n\times n}).
$$
The regularity of $J_0$ follows since $\det\colon \R^{n\times n} \rightarrow \R$ is a smooth mapping. To show the positivity of $J_0$, we use Lemma \ref{lemma:invertibility_pertubation_of_identity} below. To this end, let us show that
$$
M(t) := \sup_{(x,y) \in \widebar{\Om}\times \widebar{Y^s}} \adjnorm{\nabla_x \bfu\tx + \sum_{i,j=1}^n \bfe(\bfu)_{ij}\tx \nabla_y \bfchi_{ij}(y)}_2 < 1
$$
for all $t \in [0,T]$. Indeed, we have 
\begin{equation}\label{eq:smallness_inequality}
\begin{aligned}
M(t) &\leq   n \sup_{x \in \widebar{\Om}}\norm{\nabla_x \bfu\tx}_\infty \bigl(1+ \sum_{i,j=1}^n \underbrace{\sup_{y \in \widebar{Y^s}}\norm{\nabla_y \bfchi_{ij}(y)}_\infty}_{= C_{\bfchi}} \bigr)  \\
& \stackrel{(*)}{\leq} C_\bfu(n + n^3 C_{\bfchi})\left(\norm{\textbf{F}_e(t)}_{\mathcal{C}^{1,\sigma}(\Om)^{n \times n} }+  \norm{\bfu_D(t)}_{\mathcal{C}^{2,\sigma}(\Om)^n} \right) < 1
\end{aligned}
\end{equation}
for all $t \in [0,T]$. Here, $(*)$ follows from \eqref{eq:giaquinta_estimate} and the last inequality follows from the smallness assumption \eqref{eq:smallness_assumption}. Then, by Lemma \ref{lemma:invertibility_pertubation_of_identity}, we conclude $J_0\txy>0$ for all $\txy \in \widebar{\Om}_T \times \widebar{Y}^s$. %
Finally, to estimate $J_0$ from above, we compute by virtue of \eqref{eq:giaquinta_estimate} and \eqref{eq:smallness_assumption}:
\begin{align*}
\begin{split}
\sup_{\txy \in \widebar{\Om}_T\times \widebar{Y^s}} & \norm{\bfF_0(t,x,y)}_\infty \\
& = \sup_{\txy \in \widebar{\Om}_T\times \widebar{Y^s}}\norm{\textbf{E}_n + \nabla_x \bfu(t,x) + \sum_{i,j=1}^n \bfe(\bfu)_{ij}(t,x) \nabla_y \bfchi_{ij}(y)}_\infty\\
& \leq 1 + \sup_{\txy \in \widebar{\Om}_T\times \widebar{Y^s}} \norm{\nabla_x \bfu(t,x) + \sum_{i,j=1}^n \bfe(\bfu)_{ij}(t,x) \nabla_y \bfchi_{ij}(y)}_\infty\\
& \leq 1 + \sup_{(t,x) \in \widebar{\Om}_T} \norm{\nabla_x \bfu(t,x)}_\infty (1+n^2 C_{\bfchi})\\
& \leq 1 + \frac{1 + n^2C_{\bfchi} }{n + n^3C_{\bfchi}} = 1+\frac{1}{n}.
\end{split}
\end{align*}
Then, by \textit{Hadamard}'s inequality \cite[Corollary~7.8.2]{horn2012}, we have
$$
J_0\txy = \det(\bfF_0\txy) \leq n^{\frac{n}{2}}\left(1 + \frac{1}{n}\right)^n =: \Gamma.
$$
\item Due to i), it holds that $J_0(t,x,y) = \det(\bfF_0\txy)> 0$ for all $(t,x,y) \in [0,T]\times \widebar{\Om} \times \widebar{Y^s}$, which is equivalent to the invertibility of $\bfF_0$; hence $\bfD_0$ is well-defined. The regularity follows since the inverse of a matrix with positive determinant can be computed in terms of the determinant and the cofactor matrix which involves only smooth mappings. The symmetry of $\textbf{D}_0$ follows from the symmetry of $\widehat{\textbf{D}}$:
$$
D_{0,ij} = \sum_{k,l=1}^n J_0 F_{0,ik}^{-1}\widehat{D}_{kl}F_{0,lj}^{-T} \stackrel{\eqref{eq:symmetry_of_D_hat}}{=} \sum_{k,l=1}^n J_0 F_{0,jl}^{-1}\widehat{D}_{lk}F_{0,ki}^{-T} =D_{0,ji}, 
$$
for $i,j = 1,...,n$ and for all $(t,x,y) \in (0,T)\times \Om \times Y^s$. In order to prove the ellipticity of $\bfD_0$, let $\boldsymbol{\xi} \in \R^n$. Then, by using the ellipticity of $\widehat{\bfD}$, we get
\begin{align*}
\bfD_0\bfxi\cdot \bfxi &= J_0 \sum_{i,j,k,l=1}^n F_{0,ik}^{-1}\widehat{D}_{kl}F_{0,lj}^{-T} \xi_j \xi_i = J_0 \sum_{k,l=1}^n \widehat{D}_{kl} \Biggl(\sum_{j=1}^n F_{0,lj}^{-1} \xi_j \Biggr) \Biggl(\sum_{i=1}^n F_{0,ki}^{-1} \xi_i \Biggr)\\
& \stackrel{\eqref{eq:ellipticity_of_D_hat}}{\geq} J_0 \delta |\bfF_0^{-1}\bfxi |^2 \geq 0,
\end{align*}
for all $(t,x,y) \in (0,T)\times \Om \times Y^s$. Equality in the last step is only achieved for $\bfxi \equiv 0$ since $\bfF_0^{-1}$ is a regular matrix. Thus we have shown that $\bfD_0$ is symmetric positive definite and we can conclude
$$
\bfD_0\bfxi\cdot \bfxi \geq \lambda_{\text{min}}|\bfxi|^2
$$
with $\lambda_{\text{min}} > 0$ the smallest eigenvalue of $\bfD_0$.
\end{enumerate}
\end{proof}
For the sake of a self-contained presentation, we state the following result about the invertibility of a perturbation of the identity matrix $\textbf{E}_n$. 
%
\begin{lemma}\label{lemma:invertibility_pertubation_of_identity}
Let $\textbf{B} \in \R^{n\times n}$ with $\norm{\textbf{B}}_2 < 1$. Then $\det(\textbf{E}_n + \textbf{B}) > 0$.
\end{lemma}
\begin{proof}
The proof is adopted from \cite[Theorem 5.5-1]{ciarlet1988}. Let $\tau \in [0,1]$. Then $\textbf{E}_n + \tau \textbf{B}$ is invertible, see e.g. \cite[Section 5.7]{alt2016}, and therefore 
\begin{equation} \label{eq:invertibility_of_E+B}
\det(\textbf{E}_n + \tau \textbf{B}) \neq 0.
\end{equation}
Now, let us define
$$
\delta \colon [0,1] \rightarrow \R,\;  \delta(\tau) = \det(\textbf{E}_n + \tau \textbf{B}).
$$
We have that $\delta$ is continuous with $\delta(0)=1 >0$. If we assume that $\delta(1) < 0$, then, by \textit{Bolzano}'s intermediate zero theorem, there exists $\tau_0 \in (0,1)$ with $\delta(\tau_0) = 0$. This contradicts \eqref{eq:invertibility_of_E+B} and hence $\det(\textbf{E}_n + \textbf{B}) = \delta(1)  > 0$.
\end{proof}
Next we investigate the diffusion cell solutions $\eta_i$, $i=1,...,n$ which solve the variational diffusion cell problems \eqref{eq:variational_dcps}. We show that the regularity of the diffusion cell solutions with respect to $t$ and $x$ is inherited from the coefficient $\bfD_0$. This is crucial for the properties of the effective coefficient $\Dhom$, which is defined in terms of the diffusion cell solutions, see \eqref{eq:Dhom}.
\begin{proposition} \label{prop:existence_uniqueness_regularity_eta_i}
Suppose assumptions (\ass{A1})-(\ass{A5}) as well as the smallness assumption \eqref{eq:smallness_assumption} are satisfied. Then, for $i = 1,...,n$ and all $(t,x) \in \widebar{\Om}_T$ there exists a unique solution $\eta_i \in \mathcal{C}^1([0,T]; \C^{1,\sigma}(\widebar{\Om}; W_\text{per} (\omega))$ to problem \eqref{eq:variational_dcps}. 
\end{proposition}
\begin{proof}
The existence and uniqueness of a solution $\eta_i(t,x,\cdot_y) \in W_\text{per}(\omega)$ for all $i=1,...,n$ and $(t,x) \in \widebar{\Om}_T$ follows from the \textit{Lax}-\textit{Milgram} lemma. Moreover, we have 
\begin{equation}\label{eq:dcp_stability_estimate}
\adjnorm{\eta_i\txdy}_{H^1(Y^s)} \leq C \adjnorm{\bfD_0\txdy \bfe_i}_{L^2(Y^s)^n} \quad \forall (t,x) \in \widebar{\Om}_T
\end{equation} 
with $C>0$ independent of $t$ and $x$. We show here only the regularity of $\eta_j$ with respect to the time variable $t$. The regularity with respect to $x$ can be obtained in a similar way. Let us consider the following variational problems.
\begin{itemize}
\item Find $\eta_i\txdy \in W_\text{per}(\omega)$ such that
\begin{equation} \label{eq:pb_for_eta}
\int_{Y^s} \bfD_0(t,x,y) \nabla_y\eta_i(t,x,y) \cdot \nabla_y \zeta(y) \textrm{d}y = \int_{Y^s} \bfD_0(t,x,y) \bfe_i \cdot \nabla_y \zeta(y) \textrm{d}y 
\end{equation}
for all $\zeta \in W_\text{per}(\omega)$, $i=1,...,n$ and $(t,x) \in [0,T] \times \widebar{\Om}$.
\item Find $\eta_i(t+h,x,\cdot_y) \in W_\text{per}(\omega)$ such that
\begin{equation}\label{eq:pb_for_eta+h}
\int_{Y^s} \bfD_0(t+h,x,y) \nabla_y\eta_i(t+h,x,y) \cdot \nabla_y \zeta(y) \textrm{d}y = \int_{Y^s} \bfD_0(t+h,x,y) \bfe_i \cdot \nabla_y \zeta(y) \textrm{d}y 
\end{equation}
for all $\zeta \in W_\text{per}(\omega)$, $i=1,...,n$, $(t,x) \in (0,T) \times \widebar{\Om}$ and $h>0$ sufficiently small.
\item Given $\eta_i\txdy \in W_\text{per}(\omega)$, find $\theta_i\txdy \in W_\text{per}(\omega)$ such that
\begin{equation}\label{eq:pb_for_theta}
\begin{aligned}
\int_{Y^s} (\partial_t \bfD_0(t,x,y) \nabla_y \eta_i(t,x,y) +  \bfD_0(t,x,y)& \nabla_y \theta_i(t,x,y)) \cdot \nabla_y \zeta(y) \textrm{d}y \\
&= \int_{Y^s} \partial_t \bfD_0(t,x,y) \bfe_i \cdot \nabla_y \zeta(y) \textrm{d}y 
\end{aligned}
\end{equation}
for all $\zeta \in W_\text{per}(\omega)$, $i=1,...,n$ and $(t,x) \in [0,T] \times \widebar{\Om}$.
\end{itemize}
To establish the continuity of $\eta_i$ with respect to time, let us subtract \eqref{eq:pb_for_eta} from \eqref{eq:pb_for_eta+h} and add suitable terms. Then, the resulting problem reads: 
\begin{equation*}
\begin{aligned}
\int_{Y^s} \bfD_0(t,x,y)& \left( \nabla_y \eta_i(t+h,x,y) - \nabla_y \eta_i(t,x,y) \right) \cdot \nabla_y \zeta(y) \textrm{d}y\\
&=\int_{Y^s} \left( \bfD_0(t+h,x,y) - \bfD_0(t,x,y)\right)\left(\nabla_y \eta_i(t+h,x,y) + \bfe_i\right)\cdot \nabla_y \zeta(y) \textrm{d}y
\end{aligned}
\end{equation*}
for all $\zeta \in W_\text{per}(\omega)$, $i=1,...,n$ and $(t,x) \in (0,T) \times \widebar{\Om}$. By using an estimate similar to \eqref{eq:dcp_stability_estimate} for \eqref{eq:pb_for_eta+h} we have for all $x \in \widebar{\Om}$
\begin{align*}
\big \lVert \eta_i(t+h,x,\cdot_y) &-  \eta_i(t,x,\cdot_y) \big \rVert_{H^1(Y^s)} \\
&\leq C \adjnorm{( \bfD_0(t+h,x,\cdot_y) - \bfD_0(t,x,\cdot_y)) (\nabla_y \eta_i(t+h,x,\cdot_y) + \bfe_i )}_{L^2(Y^s)^n}\\
& \leq C \adjnorm{\norm{\bfD_0(t+h,x,\cdot_y) - \bfD_0(t,x,\cdot_y)}_2 |\nabla_y \eta_i(t+h,x,\cdot_y) + \bfe_i |}_{L^2(Y^s)} \\
& \leq C \sup_{y\in \widebar{Y^s}} \adjnorm{\bfD_0(t+h,x,\cdot_y) - \bfD_0(t,x,\cdot_y)}_2 \adjnorm{\nabla_y \eta_i(t+h,x,\cdot_y) + \bfe_i}_{L^2(Y^s)^n} \\
& \xrightarrow{h \rightarrow 0}  0,
\end{align*}
due to the continuity of $\bfD_0$ with respect to time and \eqref{eq:dcp_stability_estimate}.\par 
For the differentiability of $\eta_i$ with respect to $t$, let us subtract  \eqref{eq:pb_for_eta} from \eqref{eq:pb_for_eta+h}, divide by $h>0$, subtract \eqref{eq:pb_for_theta} and add suitable terms. We obtain
\begin{align*}
\int_{Y^s} &\bfD_0(t,x,y) \nabla_y \left(\frac{\eta_i(t+h,x,y) - \eta_i\txdy}{h} - \theta_i(t,x,y) \right)  \cdot \nabla_y \zeta(y) \textrm{d}y\\
&= \int_{Y^s} \left( \frac{\bfD_0(t+h,x,y) - \bfD_0(t,x,y)}{h} - \partial_t \bfD_0(t,x,y) \right) \left(\bfe_i - \nabla_y \eta_i(t+h,x,y) \right) \cdot \nabla_y\zeta(y)\textrm{d}y \\
& \qquad - \int_{Y^s} \partial_t \bfD_0(t,x,y) \nabla_y\left(\eta_i(t+h,x,y) -  \eta_i(t,x,y) \right) \cdot \nabla_y\zeta(y)\textrm{d}y
\end{align*}
for all $\zeta \in W_\text{per}(\omega)$, $i=1,...,n$ and $(t,x) \in (0,T) \times \widebar{\Om}$. Then, by virtue of \eqref{eq:dcp_stability_estimate}, we obtain
\begin{align*}
\Big \lVert \frac{\eta_i(t+h,x,\cdot_y) - \eta_i\txdy}{h} - \theta_i(t,x,\cdot_y) \Big \rVert_{H^1(Y^s)} \xrightarrow{h \rightarrow 0}  0,
\end{align*}
by similar estimates as above, using the regularity of $\bfD_0$ and $\eta_i$ as well as \eqref{eq:dcp_stability_estimate}. Consequently, we infer $\partial_t \eta_i = \theta_i$. The continuity of the derivative $\theta_i$ follows with similar steps as before from the continuity of $\partial_t \bfD_0$. In particular, we can extend $\eta_i$ and $\theta_i$ continuously to the boundary of the time interval. \par 
The regularity of $\eta_i$ with respect to $x$ can be shown analogously. Altogether this yields the claimed regularity $\eta_i \in \mathcal{C}^1([0,T]; \C^{1,\sigma}(\widebar{\Om}; W_\text{per} (\omega))$ for all $i=1,...,n$.
\end{proof}
Let us now investigate the properties of the effective coefficients $\Jhom$ and $\Dhom$ defined in \eqref{eq:Jhom}-\eqref{eq:Dhom} by means of $J_0$, $\bfD_0$ and the cell solutions $\eta_i$, $i=1,...,n$.
\begin{proposition}\label{proposition:properties_of_Jhom_Dhom} Suppose assumptions (\ass{A1})-(\ass{A5}) as well as the smallness assumption \eqref{eq:smallness_assumption} are satisfied. Then we have 
\begin{enumerate}
\item $\Jhom \in \mathcal{C}^1([0,T]; \C^{1,\sigma}(\widebar{\Om}))$ with $0 < \gamma^* \leq \Jhom(t,x) \leq \Gamma^*$, $\gamma^*,\Gamma^* \in \R_+$, for all $\tx \in \widebar{\Om}_T$, 
\item $\Dhom \in \mathcal{C}^1([0,T]; \C^{1,\sigma}(\widebar{\Om})^{n\times n})$ is symmetric and elliptic.
\end{enumerate}
\end{proposition}
\begin{proof}
\begin{enumerate}
\item The $\mathcal{C}^1$-$\C^{1,\sigma}$-regularity follows from the regularity of $J_0$. Due to the continuity and positivity of $J_0$, there exists $\gamma^* > 0$ such that
$$
\Jhom(t,x) \geq \gamma^* \quad \forall \tx \in \widebar{\Om}_T.
$$
On the other hand 
$$
\sup_{(t,x) \in \widebar{\Om}_T} \Jhom\tx = \sup_{(t,x) \in \widebar{\Om}_T} \int_{Y^s} J_0\txy \textrm{d}y \leq \Gamma |Y^s| := \Gamma^*
$$
with $\Gamma > 0$ from Proposition \ref{proposition:properties_of_J0_D0}.
\item $\mathcal{C}^1$-$\C^{1,\sigma}$-regularity and symmetry follow from the regularity and symmetry of $\bfD_0$ and the regularity of $\eta_i$, $i=1,...,n$. To prove the ellipticity, we follow \cite[Proposition~6.12]{cioranescu1999}. Let $\bfxi \in \R^n$ and define $w(t,x,y):= \sum_{i=1}^n \xi_i(y_i+\eta_i(t,x,y))$. We compute
\begin{align*}
\Dhom(t,x)  \bfxi \cdot \bfxi 
& = \sum_{i,j,k,l = 1}^n \int_{Y^s} D_{0,kl}(t,x,y) \xi_j\partial_l(y_j+\eta_j(t,x,y) \xi_i \partial_k(y_i+\eta_i(t,x,y)) \textrm{d}y\\
& \geq \lambda_{\text{min}}\norm{\nabla_y w(t,x,\cdot_y)}^2_{L^2(Y^s)^n} \geq 0
\end{align*}
for all $(t,x) \in \widebar{\Om}_T$ using the ellipticity of $\bfD_0$ shown in Proposition \ref{proposition:properties_of_J0_D0}. It remains to show that equality in the last step is achieved only for $\bfxi \equiv 0$. To this end, assume equality to hold for $\bfxi \neq 0$. Then, $w$ is independent of $y$ and thus there exists $c = c(t,x)$ such that
$$
c(t,x) = w(t,x,y) = \sum_{i=1}^n \xi_i(y_i+\eta_i(t,x,y))
$$
which is equivalent to 
$$
\sum_{i=1}^n  \xi_i \eta_i(t,x,y) = c(t,x) - \sum_{i=1}^3 \xi_i y_i.
$$
This is a contradiction since the right-hand side is not periodic with respect to $y$ for $\bfxi \neq 0$. Thus we have shown that $\Dhom(t,x)$ is symmetric positive definite for all $(t,x) \in \widebar{\Omega}_T$ and therefore
$$
\Dhom(t,x) \bfxi \cdot \bfxi \geq \lambda^*_\text{min}(t,x) |\bfxi|^2
$$
for all $\bfxi \in \R^n$ and all $(t,x) \in \widebar{\Omega}_T$ with $\lambda^*_\text{min}(t,x) > 0$ the smallest eigenvalue of $\bfD^*(t,x)$.
\end{enumerate}
\end{proof}
Based on the properties of the homogenized coefficients from Proposition \ref{proposition:properties_of_Jhom_Dhom}, the proof of Theorem \ref{theorem:main_theorem_2} is now standard and follows by application of \textit{Galerkin}'s method and standard regularity results for weak solutions, see e.g. \cite[Chapter 7]{evans2010}. 
\begin{remark}
The global existence for the model investigated within this chapter critically depends on the smallness assumption \eqref{eq:smallness_assumption} for the data of the macroscopic elasticity problem. The smallness assumption is given in terms of the constants $C_\bfu = C_\bfu(\Omega, \sigma,\Ahom)$ and $C_{\bfchi} = C_{\bfchi}(Y^s, \textbf{A})$. Especially $C_{\bfchi}$ depends on the microscopic geometry $Y^s$ and the elastic properties of the medium. Condition \eqref{eq:smallness_assumption} shows that small values of $C_{\bfchi}$ allow for more flexibility in the choice of the data $\bfF_e$ and $\bfu_D$.  In the following, we study the sensitivity of $C_{\bfchi}$ with respect to $Y^s$ and $\textbf{A}$ by means of numerical simulations and identify critical settings which lead to small values for $C_{\bfchi}$.
\end{remark}
%
\subsection{Sensitivity of the constant $C_{\bfchi}$ with respect to microscopic geometry and material parameters}
\label{subseq:characterization}
Let us consider  standard periodicity cells $Y^s = Y^s_{(w_1,w_2)} \subset [0,1]^2$ having the shape of a cross, defined by two cross bars with widths $w_1,w_2 \in (0,1)$, given by
\begin{equation}\label{eq:Ys-def}
Y^s_{(w_1,w_2)} = \left[ \left( \frac{1-w_1}{2}, \frac{1+w_1}{2} \right) \times (0,1) \right] \cup \left[ \left( (0,1) \times \frac{1-w_2}{2}, \frac{1+w_2}{2} \right) \right]
\end{equation}
and illustrated in Figure \ref{fig:unit_cell}.
\begin{figure}
\centering
\includegraphics[width=0.5\textwidth]{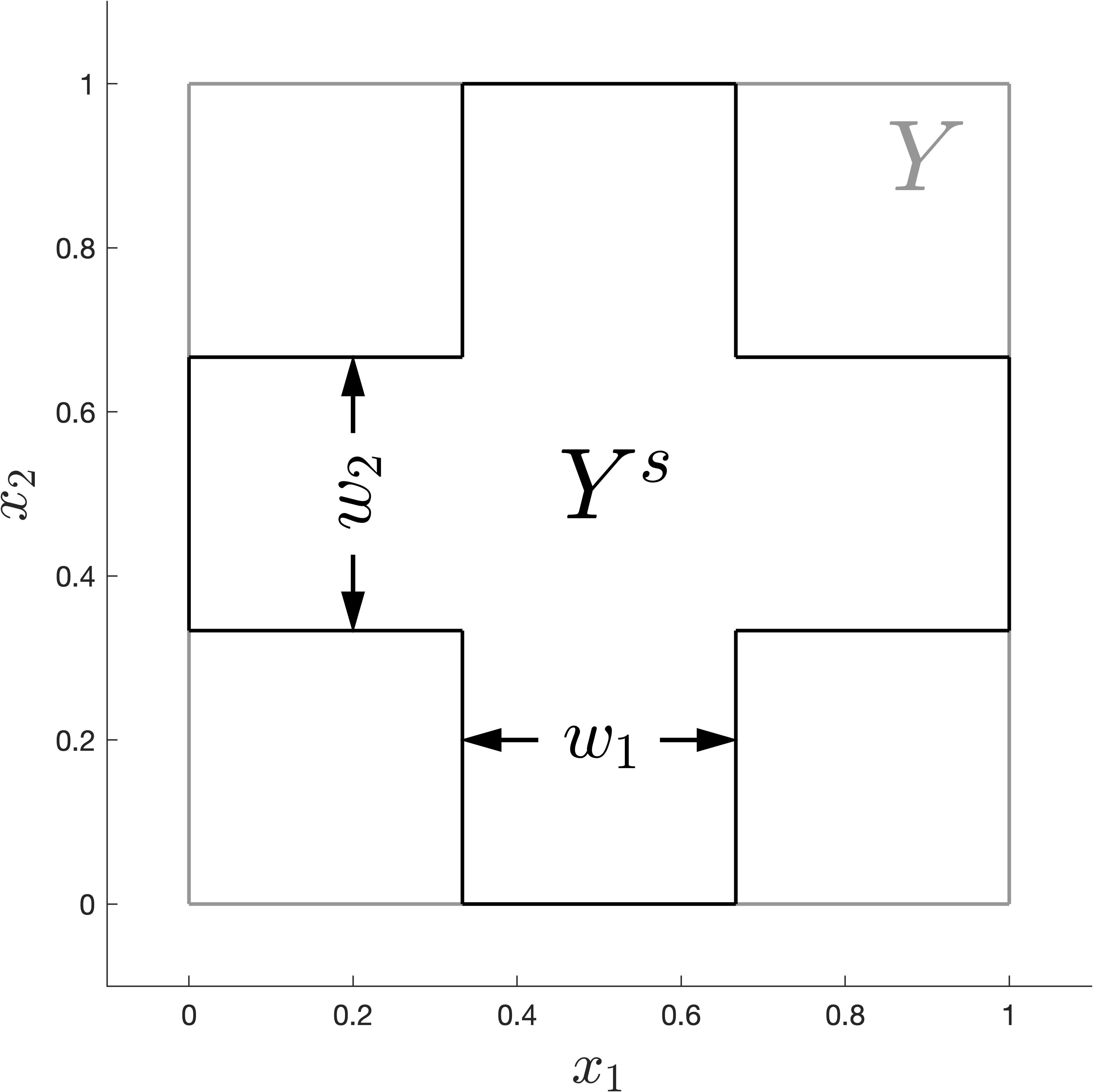}
\caption{The unit cell $Y^s$ given by \eqref{eq:Ys-def} with $w_1 = w_2 = \frac{1}{3}$.}
\label{fig:unit_cell}
\end{figure}
Furthermore, let us consider an isotropic and homogeneous medium described by the constant elasticity tensor $\textbf{A}$ with components %
\begin{equation}\label{eq:lame_representation}
A_{ijkl} = \lambda\delta_{ij}\delta_{kl} + \mu \left( \delta_{ik} \delta_{jl} + \delta_{il}\delta_{jk} \right),
\end{equation}
with \textit{Lamé} constants $\mu,\lambda>0$, for $i,j,k,l = 1,2$.

To investigate the sensitivity of the constant 
\begin{equation}\label{eq:def_C_chi_sens}
C_{\bfchi} = C_{\bfchi}(Y^s, \textbf{A}) = \sup_{y \in \widebar{Y^s}, i,j=1,...,n} \adjnorm{\nabla_y \bfchi_{ij}(y)}_\infty
\end{equation}
with respect to the geometry of the solid cell domain $Y^s$, we solved the elasticity cell problems \eqref{eq:elasticity_cell_problems} on a range of different cross-shaped domains, covering all possible crossbar values $w_1,w_2 \in (0,1)$. In all cases, we used the \textit{Lamé} parameters $\lambda = \mu = 1$. Similarly, we investigated the dependence of $C_{\bfchi}$ on $\textbf{A}$ by solving the elasticity cell problems for different values of $\lambda, \mu \in [1,20]$ on the fixed domain $Y^s_{(1/3,1/3)}$. From the resulting cell solutions, we computed $C_{\bfchi}$ according to \eqref{eq:def_C_chi_sens} and generated the plots in Figures \ref{fig:C_chi_dependence:a} and \ref{fig:C_chi_dependence:b}.\par 
For fixed \textit{Lamé} parameters, we  obtain the highest values of $C_{\bfchi}$ when the cross is constituted by one very thick bar and one very thin bar, while $C_{\bfchi}$ attains its smallest values if either both bars are thick or both are thin. In our test, the minimum $C_{\bfchi} = 0.2845$ is attained for $w_1=w_2=0.99$, which were the thickest crossbar widths considered in the simulations. The two thinnest widths considered here, $w_1=w_2=0.01$, result in local minimum for $C_{\bfchi}$, where the value attained lies slightly above the value for $w_1=w_2=0.99$. Concerning the dependence of $C_{\bfchi}$ on $\lambda, \mu$, we observe that the combination of large $\lambda$ and small $\mu$ results in the highest values for $C_{\bfchi}$ while the minimum $C_{\bfchi} = 1.4739$ is attained for $\lambda=1$ and $\mu = 20$.\par 
Combining our findings, we expect that the choice of thick crossbars, which corresponds to small perforations, together with a small value for $\lambda$ and a large one for $\mu$ produces the smallest constant $C_{\bfchi}$ which allows for larger data $\bfF_e$ and $\bfu_D$. In this regard, we obtain $C_{\bfchi} = 0.1953$ for $w_1 = w_2 = 0.99$, $\lambda = 1$ and $\mu = 20$. Increasing $\mu$ even further does however not have a significant effect on $C_{\bfchi}$, i.e. it seems that $C_{\bfchi}$ cannot be decreased further, at least in the present setting.
\begin{figure}
\centering
\begin{subfigure}[t]{0.49\textwidth}
\includegraphics[width = \textwidth]{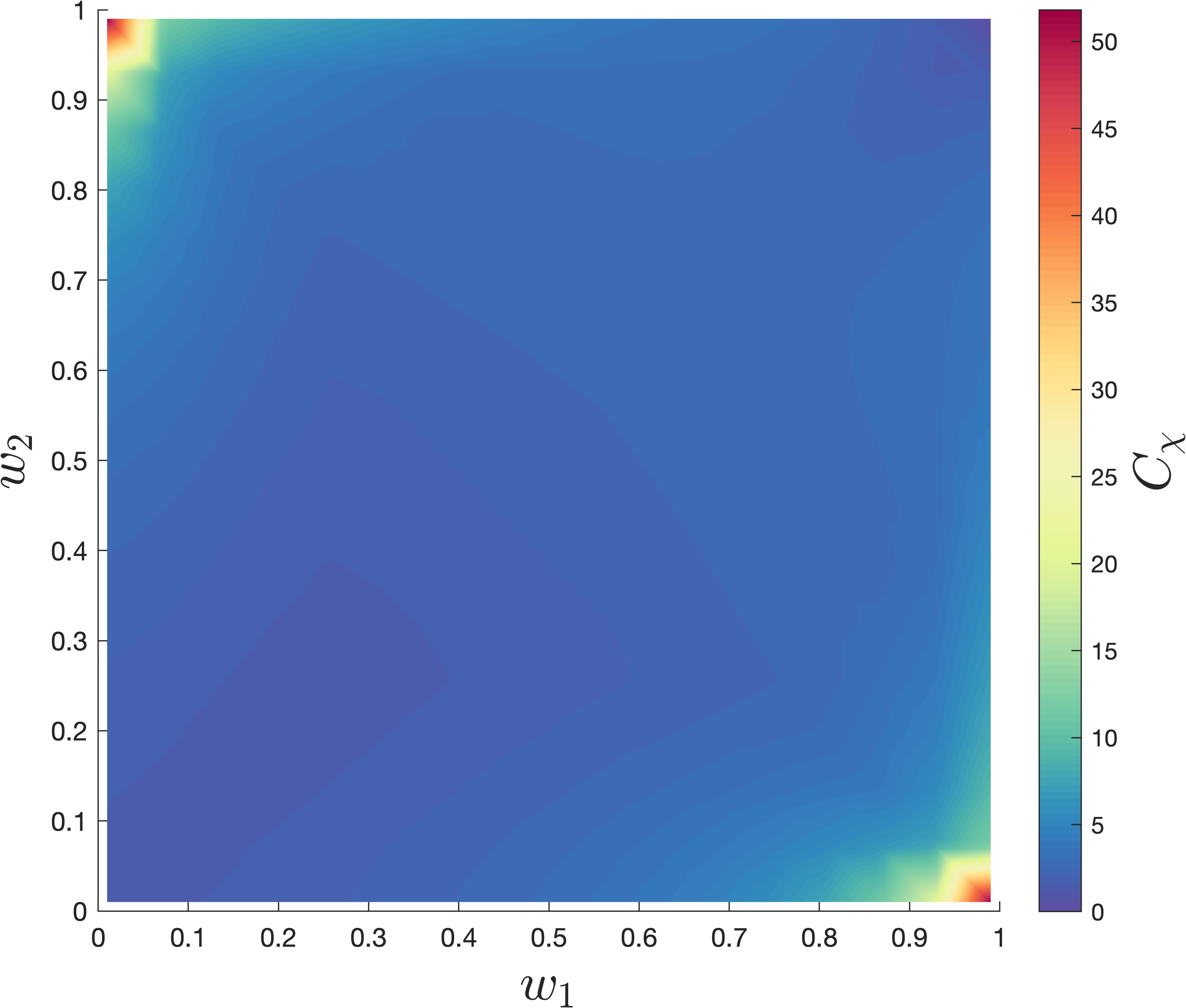}
\caption{$C_{\bfchi}$ in dependence of crossbar widths $w_1$, $w_2 \in (0,1)$ for fixed \textit{Lamé} parameters $\lambda = \mu = 1$.}
\label{fig:C_chi_dependence:a}
\end{subfigure}
\hfill
\begin{subfigure}[t]{0.49\textwidth}
\includegraphics[width = \textwidth]{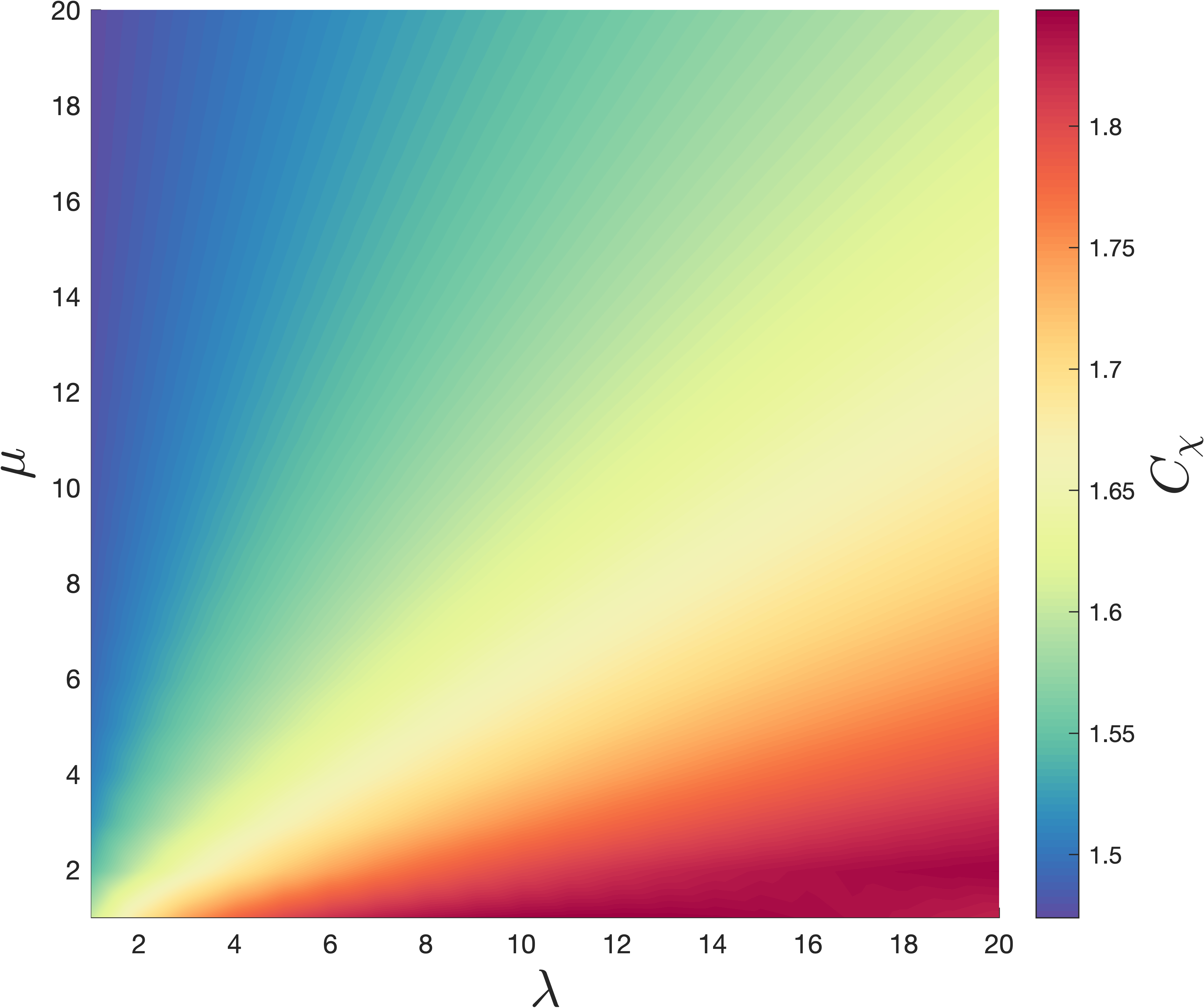}
\caption{$C_{\bfchi}$ in dependence of \textit{Lamé} parameters $\lambda, \mu \in  [1,20]$ for fixed solid domain $Y^s_{(1/3,1/3)}$.}
\label{fig:C_chi_dependence:b}
\end{subfigure}
\caption{Investigation of the dependence of $C_{\bfchi}$ on the micro-geometry and the material parameters.}
\label{fig:C_chi_dependence}
\end{figure}
\section{Numerical justification of the effective micro-macro model}%
\label{sec:numerical_justification}
\begin{algorithm}[t]
\begin{algorithmic}[1]
\State Discretize microscopic domain $Y^s$ and macroscopic domain $\Om$
\State Discretize time interval $[0,T]$ using step size $\Delta t = T/N$, $N \in \mathbb{N}$
\State Solve elasticity cell problems \eqref{eq:elasticity_cell_problems}
\State Compute $\Ahom$ according to \eqref{eq:Ahom}
\State Set $k \leftarrow 0, t_k \leftarrow 0$
\While{$t_k \leq T$}
\State Solve macroscopic elasticity problem \eqref{eq:effective_micro_macro_model:a}-\eqref{eq:effective_micro_macro_model:b} for current displacement
\For{each quadrature point in the macroscopic grid}
\State Compute $J_0$, $D_0$ according to \eqref{def:J_0}, \eqref{def:D_0}
\State Solve diffusion cell problems \eqref{eq:diffusion_cell_problems}
\State Compute $\Jhom, \Dhom$ according to \eqref{eq:Jhom}, \eqref{eq:Dhom}
\EndFor
\State Solve macroscopic transport problem \eqref{eq:effective_micro_macro_model:c}-\eqref{eq:effective_micro_macro_model:e} for current concentration
\State Set $k \leftarrow k+1, t_k \leftarrow k\Delta t$
\EndWhile
\end{algorithmic}
\caption{Numerical scheme to compute the solution to the effective micro-macro problem \eqref{eq:effective_micro_macro_model}-\eqref{def:F_0}.}
\label{alg:1}
\end{algorithm}
In this section, we show numerically the convergence of the microscopic solutions of \eqref{eq:micro_model_on_reference_domain} to the solution of the effective micro-macro model \eqref{eq:effective_micro_macro_model}-\eqref{def:F_0}, for smaller and smaller values of the scale parameter $\eps>0$, for the computational scenario described below. This numerical justification of the formally derived effective model is particularly important, as rigorous analytical convergence proofs or error estimates are not available so far. \par
For the computations we use the computational framework developed in \cite{knoch2023} and based on the open source finite element library \textit{deal.II} \cite{dealii2023,dealii2019}. We consider the case of $n=2$ space dimensions and use the classical \textit{Galerkin} method with \textit{Lagrangian} finite elements of order one for spatial discretization and the \textit{Crank}-\textit{Nicolson} method for temporal discretization. The numerical scheme used for the effective micro-macro problem is given in Algorithm \ref{alg:1}. \par 
Let us introduce the simulation scenario. We use a rectangular reference domain $\Omega$ and for the elasticity problem, we assume clamped upper and lower (outer) boundary and parabola-shaped time-periodic lateral stretching-relaxing cycles. For the transport problem, we have a homogeneous \textit{Dirichlet} boundary condition on the whole (outer) boundary, and assume that the substance enters the domain via a constant source term. The initial concentration is set to zero. More specifically, we use the computational reference domain
$$
\Omega = \left( -0.5, 0.5 \right)^2.
$$
The computational standard cell $Y^s = Y^s_{(w_1,w_2)}$, for $(w_1, w_2) \in (0,1)^2$, is given by the cross-shaped domain defined in \eqref{eq:Ys-def} and illustrated in Figure \ref{fig:unit_cell}.
In this section, we use $w_1 = w_2 = \frac{1}{3}$ for the computations.
The $\eps$-dependent computational microscopic reference domain $\Oes$ is given by
$$
\Oes =  \Om \cap \eps \left( \omega + \begin{pmatrix}\frac{1}{2} \vspace{0.3em} \\ \frac{1}{2} \end{pmatrix} \right)
$$
where $\omega$ is the unbounded periodic domain
$$
\omega = \text{int} \left( \bigcup_{\textbf{k} \in \mathbb{Z}^n} \left( \widebar{Y^s} + \textbf{k} \right) \right).
$$
\begin{figure}
\centering
\begin{subfigure}{0.49\textwidth}
    \centering
    \includegraphics[width=\textwidth]{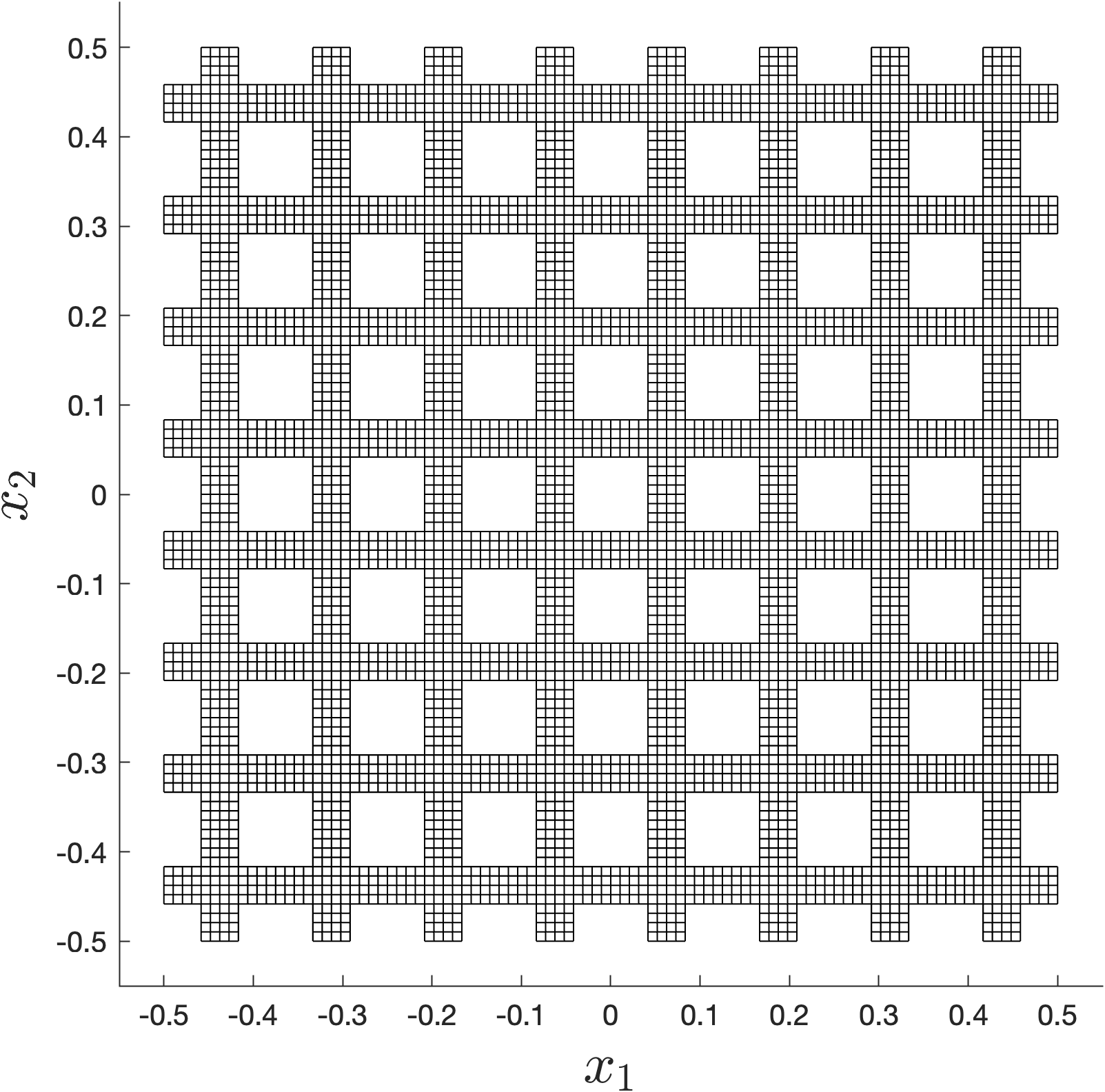}
\end{subfigure}
\hfill
\begin{subfigure}{0.49\textwidth}
    \centering
    \includegraphics[width=\textwidth]{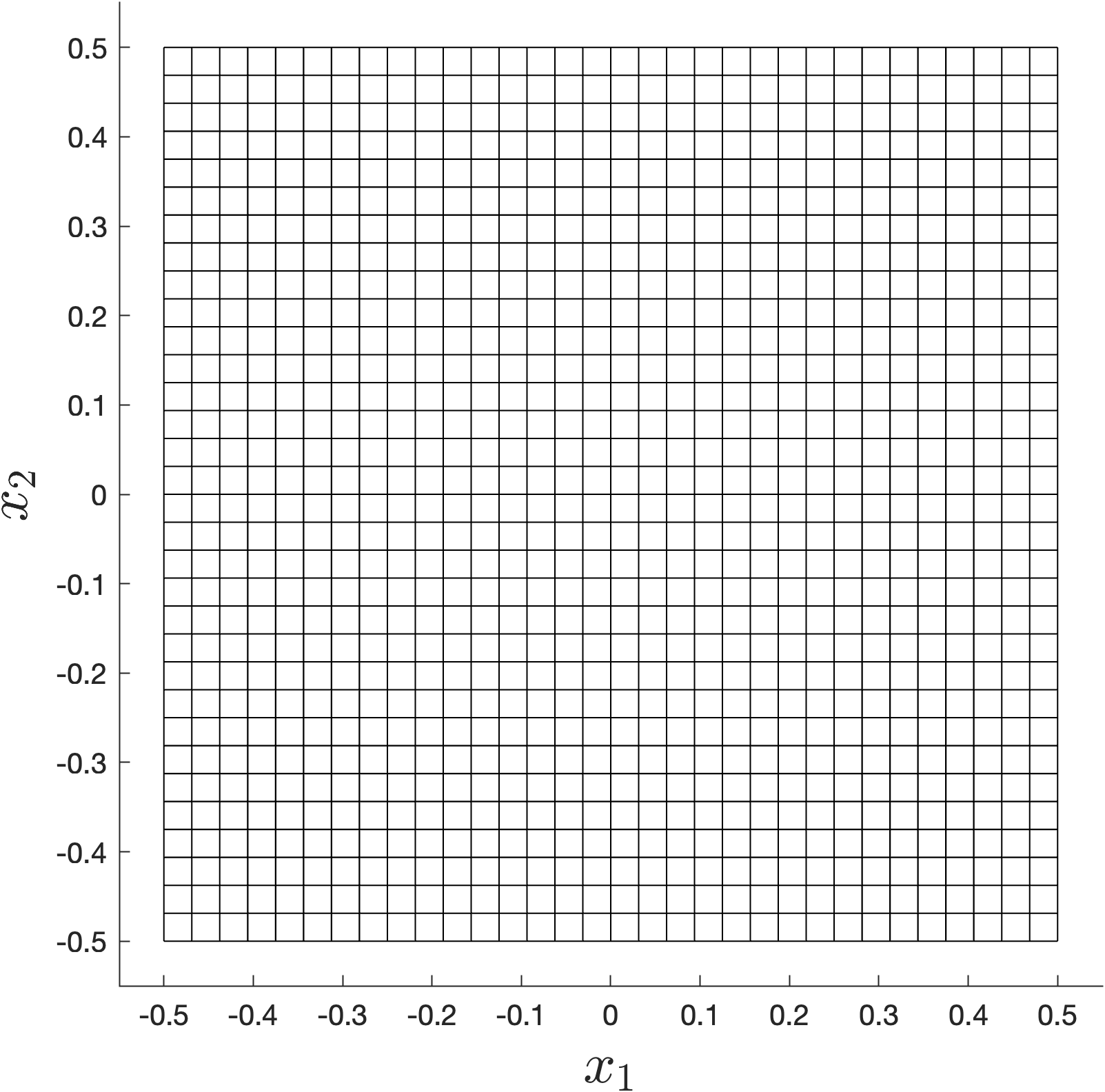}
\end{subfigure}
\caption{Finite element mesh for the perforated domain $\Oes$ (for $\eps = \frac{1}{8}$) and the macroscopic domain $\Om$.}
\label{fig:grids}
\end{figure}%
A graphical representation of the computational domains  $\Oes$ (for $\eps = \frac{1}{8}$) and $\Omega$ together with the corresponding finite element mesh can be found in Figure \ref{fig:grids}.\par 
We consider an isotropic and homogeneous medium described by the constant elasticity tensor $\textbf{A}$ with components given in \eqref{eq:lame_representation}
with \textit{Lamé} constants $\lambda =  \mu = 1$ for $i,j,k,l = 1,2$. We set 
$$
\bff_e(t,x) = (0,0)^T
$$
for all simulations. Moreover, we consider the case of a single species, i.e. $N_c = 1$, with diffusion tensor 
\begin{equation*}
\wh{\bfD} = \begin{pmatrix} 0.5 & 0 \\ 0 & 0.5 \end{pmatrix},
\end{equation*}
and use 
\begin{equation}\label{eq:rhs_and_init_value}
f_d^1  = 1 \quad \text{and} \quad c^{1,0} = 0.
\end{equation}
Let us recall that we always assume homogeneous \textit{Neumann} boundary conditions at the inner part of the boundary for the microscopic problem, see \eqref{eq:micro_model_on_reference_domain:b} and \eqref{eq:micro_model_on_reference_domain:e}. For the elasticity subproblem, homogeneous \textit{Dirichlet} boundary conditions are prescribed at the upper and lower part of the outer boundary and a time-periodic, parabola-shaped displacement at the lateral outer boundaries is assumed:
\begin{gather*}
\bfu_D(t,x) = \begin{cases} \left( \frac{a}{2} (1 - \cos(2\pi ft))(1- 4x_2^2), 0\right)^T, x \in \partial \Om \cap \{x_1 = 0.5\}, \\
\left( -\frac{a}{2} (1 - \cos(2\pi ft))(1- 4x_2^2), 0\right)^T, x \in\partial \Om \cap \{x_1 = -0.5\},\\
(0,0)^T, x \in \partial \Om \cap \{x_2 = \pm0.5\}, \end{cases}
\end{gather*}
with amplitude $a=0.1$ and frequency$f=1$. Note that this function is restricted to the outer boundary of the perforated domain in the case of the microscopic problem. \par \bigskip
We compute the solutions to the effective micro-macro problem \eqref{eq:effective_micro_macro_model}-\eqref{def:F_0} and to the microscopic problem \eqref{eq:micro_model_on_reference_domain} (for $\eps \in \{1,\frac{1}{2},\frac{1}{4},\frac{1}{8},\frac{1}{16},\frac{1}{32}\}$). 
We investigate the convergence of the microscopic displacement and concentration to the macroscopic ones in the $L^2$- and the $H^1$-norm at a fixed time $t_\text{eval}=1.5$. For the $H^1$-norm, we include also the first order corrector terms, which can be computed via \eqref{eq:correctors}. We quantify the numerical order of convergence via the formulas
\begin{equation}\label{eq:EOC_L2}
EOC_i = \log_2\left( \frac{\| \square - \square_{\eps = 2^{-(i-1)}} \|_{L^2(\Oes)} }{\| \square - \square_{\eps = 2^{-i}} \|_{L^2(\Oes)}} \right),
\end{equation}
and
\begin{equation}\label{eq:EOC_H1}
EOC_i = \log_2\left( \frac{\| \square + 2^{-(i-1)}\square_1 - \square_{\eps = 2^{-(i-1)}} \|_{H^1(\Oes)} }{\| \square + 2^{-i}\square_1 - \square_{\eps = 2^{-i}} \|_{H^1(\Oes)}} \right),
\end{equation}
where $i\in \{1, 2, 3,4,5\}$, and $\square \in \{\bfu,c\}(t=t_\text{eval},\cdot)$. For example, the EOC of the concentration in the $H^1$-norm is obtained by computing
$$
EOC_i = \log_2\left( \frac{\| c(1.5,\cdot)  + 2^{-(i-1)}c_1(1.5,\cdot,\frac{\cdot}{\eps})  - c_{\eps = 2^{-(i-1)}}(1.5,\cdot)  \|_{H^1(\Oes)}}{\| c(1.5,\cdot) + 2^{-i}c_1(1.5,\cdot,\frac{\cdot}{\eps}) - c_{\eps = 2^{-i}}(1.5,\cdot)  \|_{H^1(\Oes)}} \right).
$$
\begin{table}
\begin{subtable}[c]{\textwidth}
  \centering
  \begin{tabular}{r | c r | c r}\toprule
$\eps^{-1}$ & $\norm{\bfu_\eps - \bfu}_{L^2(\Oes)^2}$ & EOC & $\norm{\bfu_\eps - (\bfu + \eps \bfu_1)}_{H^1(\Oes)^2}$ & EOC   \\
\midrule
	  1 & 5.37900e-03 &    -   &  1.12561e-01  &    -    \\
      2 & 5.37829e-03 & 0.00 &  9.62039e-02  & 0.23  \\
      4 & 3.17833e-03 & 0.76 &  6.58677e-02 & 0.55   \\
      8 & 1.88574e-03 & 0.75 &  4.76079e-02 & 0.47   \\
     16 & 1.05914e-03 & 0.83 & 3.51168e-02 & 0.44   \\
     32 & 5.65785e-04 & 0.90 & 2.71631e-02  & 0.37 \\
\bottomrule
\end{tabular}
  \subcaption{displacement}
      \label{tab:eps_quasi_static_convergence:a}
  \end{subtable}
\vfill
\begin{subtable}[c]{\textwidth}
\centering
  \begin{tabular}{r|c r | c r } \toprule
  $\eps^{-1}$ & $\norm{c_\eps - c}_{L^2(\Oes)}$ & EOC & $\norm{c_\eps - (c + \eps c_1)}_{H^1(\Oes)}$ & EOC  \\ \midrule
   1 &  1.46398e-02 & -      & 2.02208e-01 &     -     \\
   2 &  1.12154e-02 & 0.38 & 1.37726e-01 & 0.55   \\
   4 & 6.94675e-03 & 0.69 & 7.22432e-02 & 0.93   \\
   8 & 3.38927e-03 & 1.04 & 4.78495e-02 & 0.59   \\
 16 & 1.65630e-03 & 1.03 & 1.48408e-02 & 1.69  \\
 32 & 9.91574e-04 & 0.74 & 1.13611e-02 & 0.39   \\
 \bottomrule
\end{tabular}
  \subcaption{concentration}
    \label{tab:eps_quasi_static_convergence:b}
  \end{subtable}
  \caption{Numerical convergence study for decreasing $\varepsilon$. The solutions are evaluated at $t_\text{eval} = 1.5$ and the estimated orders of convergence (EOC) is computed according to \eqref{eq:EOC_L2} and \eqref{eq:EOC_H1}.}
  \label{tab:eps_quasi_static_convergence}
\end{table}%
\begin{figure}
\centering
\begin{subfigure}[b]{0.49\textwidth}
\centering
\includegraphics[width=\textwidth]{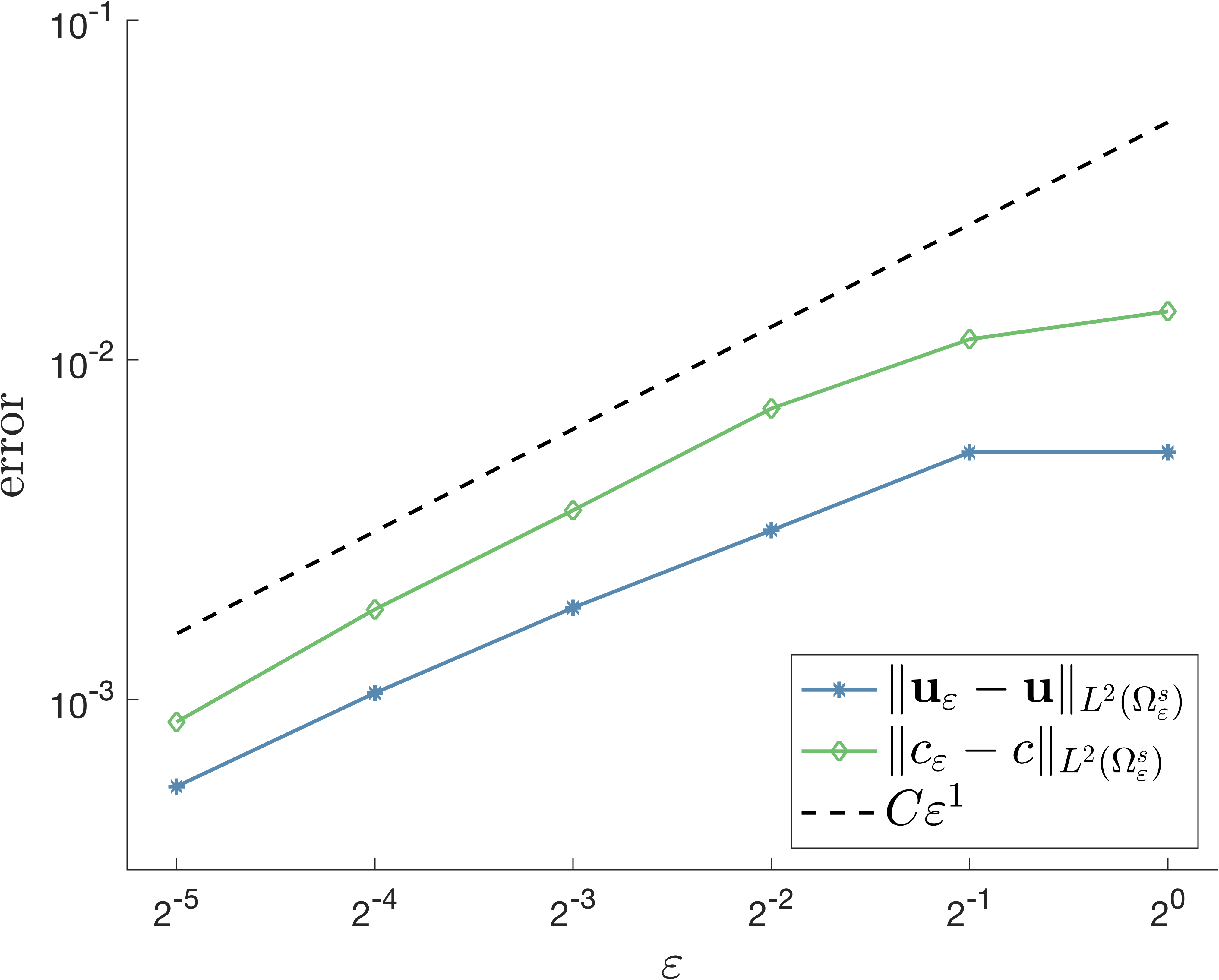}
\end{subfigure}
\hfill
\begin{subfigure}[b]{0.49\textwidth}
\centering
\includegraphics[width=\textwidth]{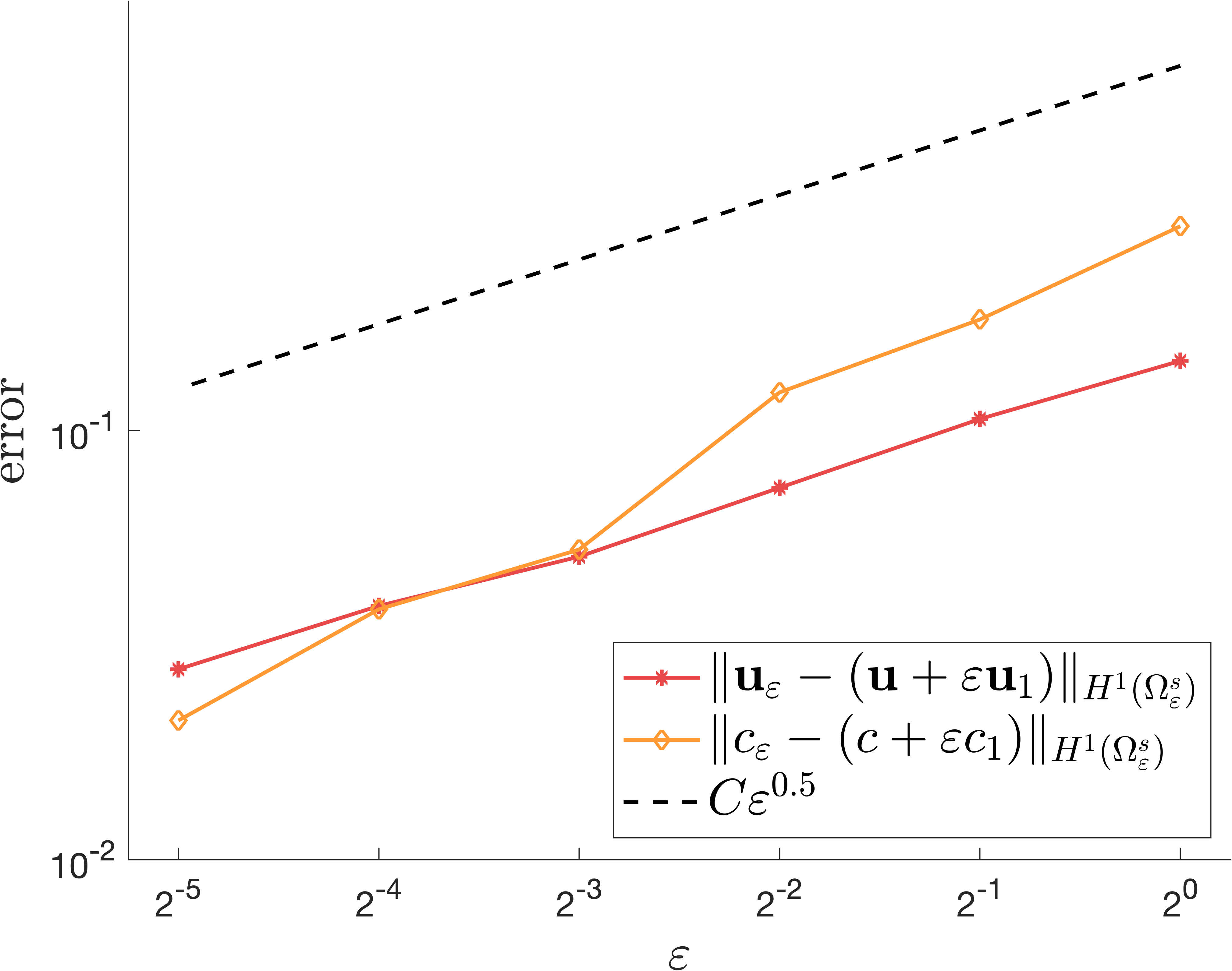}
\end{subfigure}
\caption{Log-log plots generated from the convergence data of Table \ref{tab:eps_quasi_static_convergence} assessing the convergence of the microscopic solutions to the macroscopic ones in the $L^2$-norm (left) and the $H^1$-norm (right).}
\label{fig:log_log_convergence_plots}
\end{figure}%
The so-obtained convergence results are collected in Table \ref{tab:eps_quasi_static_convergence} and visualized using log-log plots in Figure \ref{fig:log_log_convergence_plots}. In all cases, the error becomes smaller for smaller values of $\eps$ and the order of convergence is higher for the $L^2$ norm than for the $H^1$ norm, for both the displacement and the concentration. From the EOC values, we deduce that the error in the $L^2$-norm converges approximately with order $1$ and the error in the $H^1$-norm, when the respective corrector is included, with order $\frac{1}{2}$. We still observe some dispersion in the EOC-values, which might be attributed to the accuracy losses due to spatial and temporal discretization. Let us remark that no convergence can be observed in the $H^1$-norm when the corrector terms \eqref{eq:correctors} are not included (data not shown). \par
Eventually, let us supplement the above investigations for the convergence in the $L^2$- and $H^1$-norm, evaluated at a fixed time point $t_\text{eval} = 1.5$, by time series of the microscopic and effective concentrations. To this end, we compute in each time step the mean value of the concentration which is given by 
$$
\frac{1}{|\Oes|} \int_{\Oes} c_\eps(t,x) \textrm{d}x \quad \text{ and } \quad \frac{1}{|\Om|} \int_{\Om} c(t,x) \textrm{d}x,
$$
respectively. The so-obtained values are plotted over time in Figure \ref{fig:convergence_of_mean}. The mean value of the microscopic concentration is depicted for $\eps \in \{1,\frac{1}{2},\frac{1}{4},\frac{1}{8}\}$. Clearly, the curves of the mean value computed for $c_\eps$ approach the curve of the macroscopic mean concentration. In particular, we observe that it is hard to optically distinguish the curves for $\eps = \frac{1}{4},\frac{1}{8}$ from the orange curve of the macroscopic mean concentration, indicating fast convergence. \par \bigskip
\begin{figure}
\centering
\includegraphics[width=0.75\textwidth]{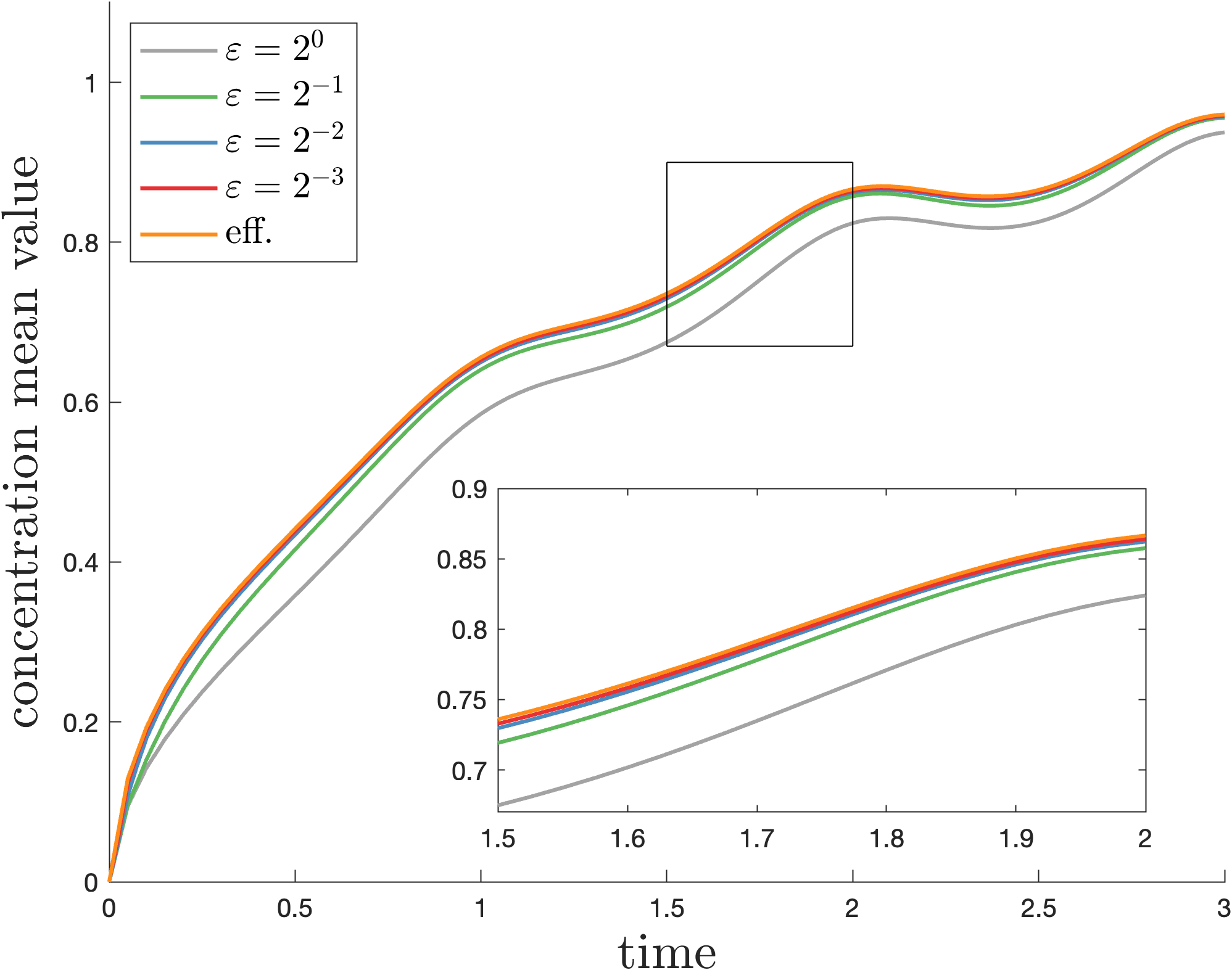}
\caption{Time evolution of the concentration mean value for the microscopic model with different values of $\eps$ and for the effective micro-macro model.}
\label{fig:convergence_of_mean}
\end{figure}
To conclude this section, we present numerical simulations of the spatial distribution of the approximation errors
$$
\bfue(t_\text{eval},\cdot)-\bfu\arrowvert_{\Oes}(t_\text{eval},\cdot), \quad c_\eps(t_\text{eval},\cdot)-c\arrowvert_{\Oes}(t_\text{eval},\cdot),
$$
and
\begin{gather*}
 \bfue(t_\text{eval},\cdot)-\left( \bfu\arrowvert_{\Oes}(t_\text{eval},\cdot) + \eps \bfu_1\left(t_\text{eval},\cdot, \frac{\cdot}{\eps}\right) \right), \\
  c_\eps(t_\text{eval},\cdot) - \left(c\arrowvert_{\Oes}(t_\text{eval},\cdot)+ \eps c_1\left(t_\text{eval},\cdot, \frac{\cdot}{\eps}\right)\right)
\end{gather*}
evaluated on the microscopic perforated domain, at a fixed time point $t_\text{eval}$. We also illustrate the correctors $\bfu_1$ and $c_1$, see \eqref{eq:correctors}, which contain information about the oscillations of the microscopic solutions $\bfue$ and $c_\eps$ which are not covered by the macroscopic solutions $\bfu$ and $c$. 
\begin{figure}
\centering
\begin{subfigure}[b]{0.49\textwidth}
\centering
\includegraphics[width=0.8\textwidth]{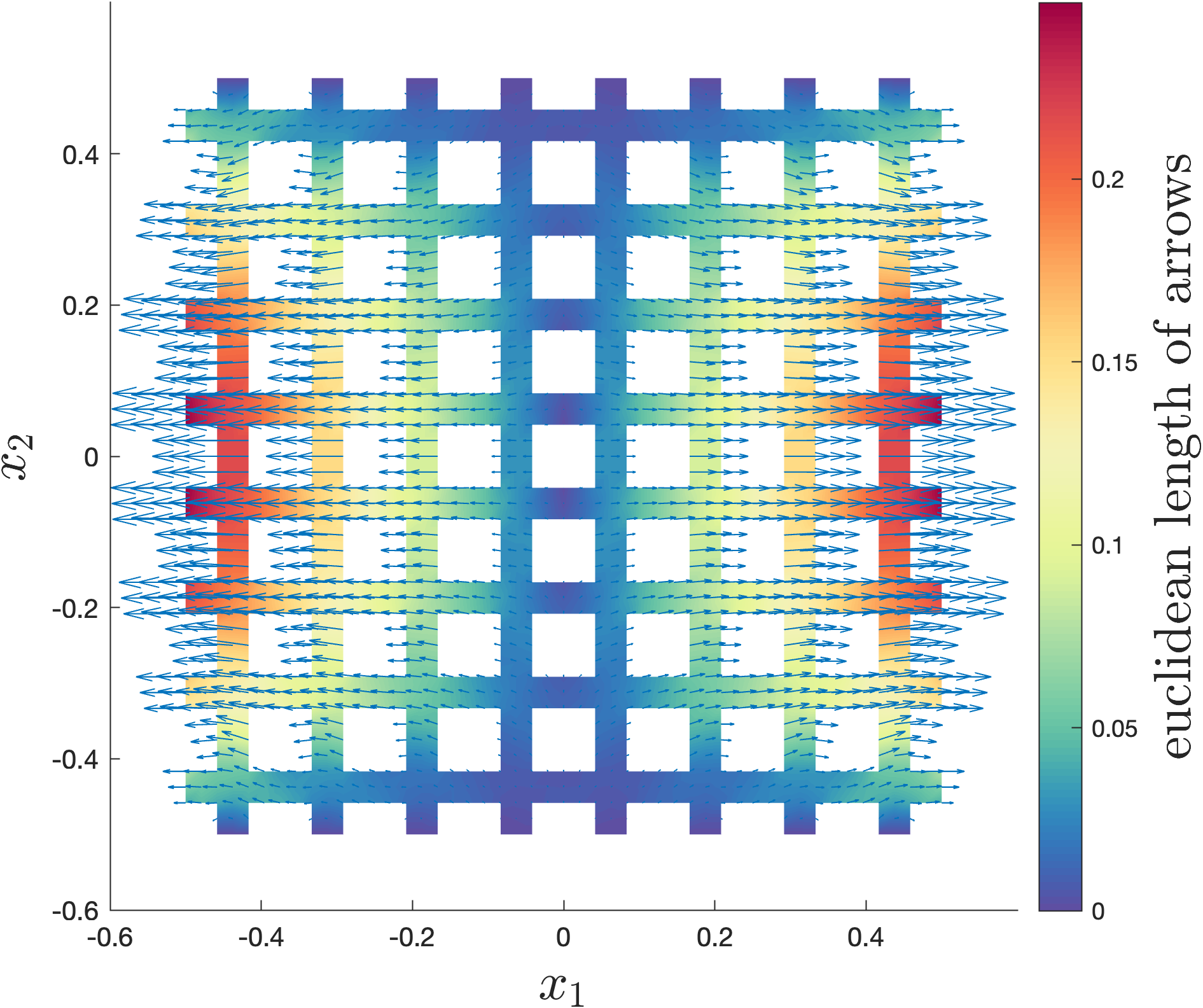}
\caption{$\bfu_{\eps}(t_\text{eval},\cdot)$}
\label{fig:corrector_plots:elasticity:a}
\end{subfigure}
\hfill
\begin{subfigure}[b]{0.49\textwidth}
\centering
\includegraphics[width=0.8\textwidth]{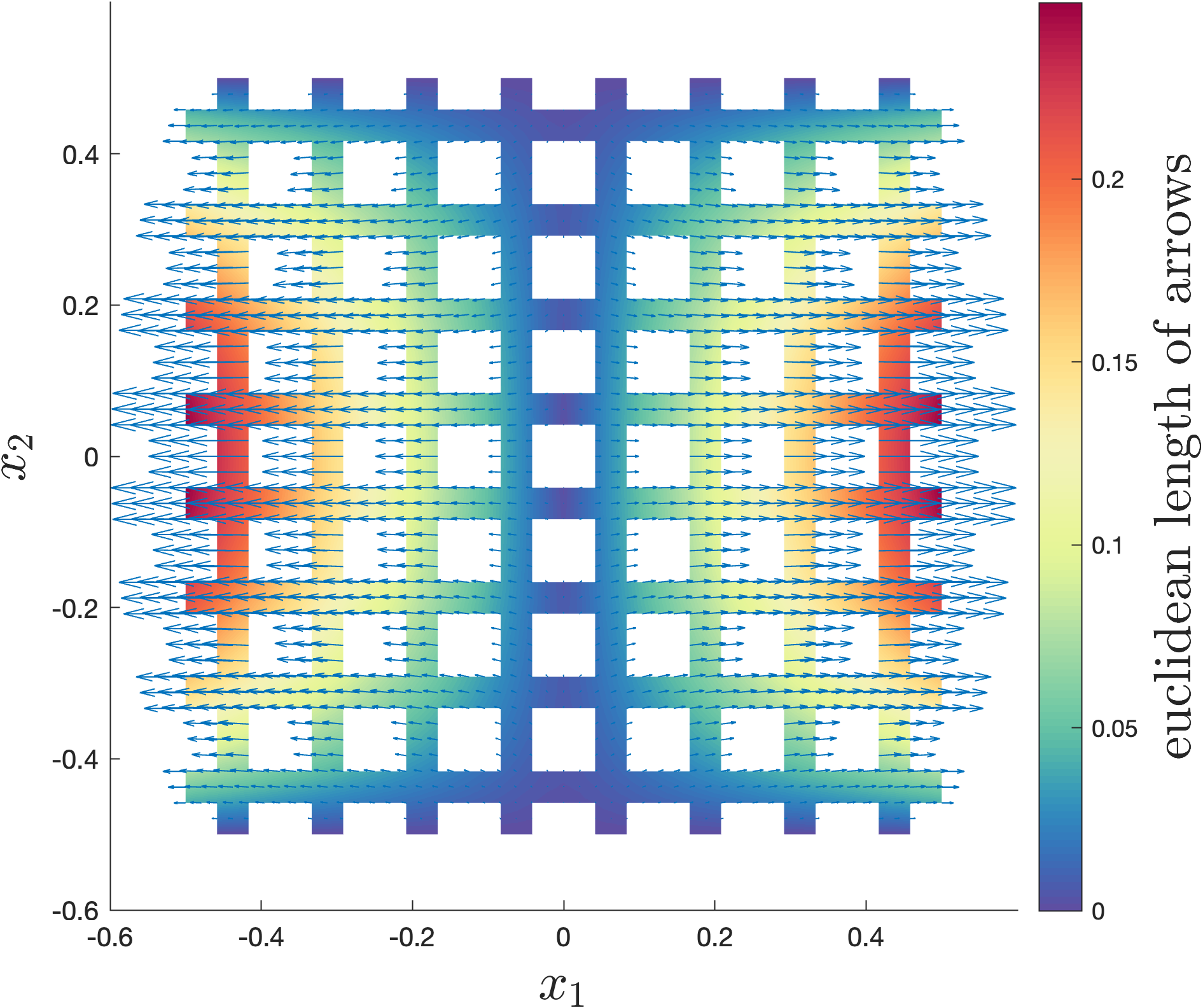}
\caption{$\bfu\arrowvert_{\Oes}(t_\text{eval},\cdot)$}
\label{fig:corrector_plots:elasticity:b}
\end{subfigure}
\hfill
\begin{subfigure}[b]{0.49\textwidth}
\centering
\includegraphics[width=0.8\textwidth]{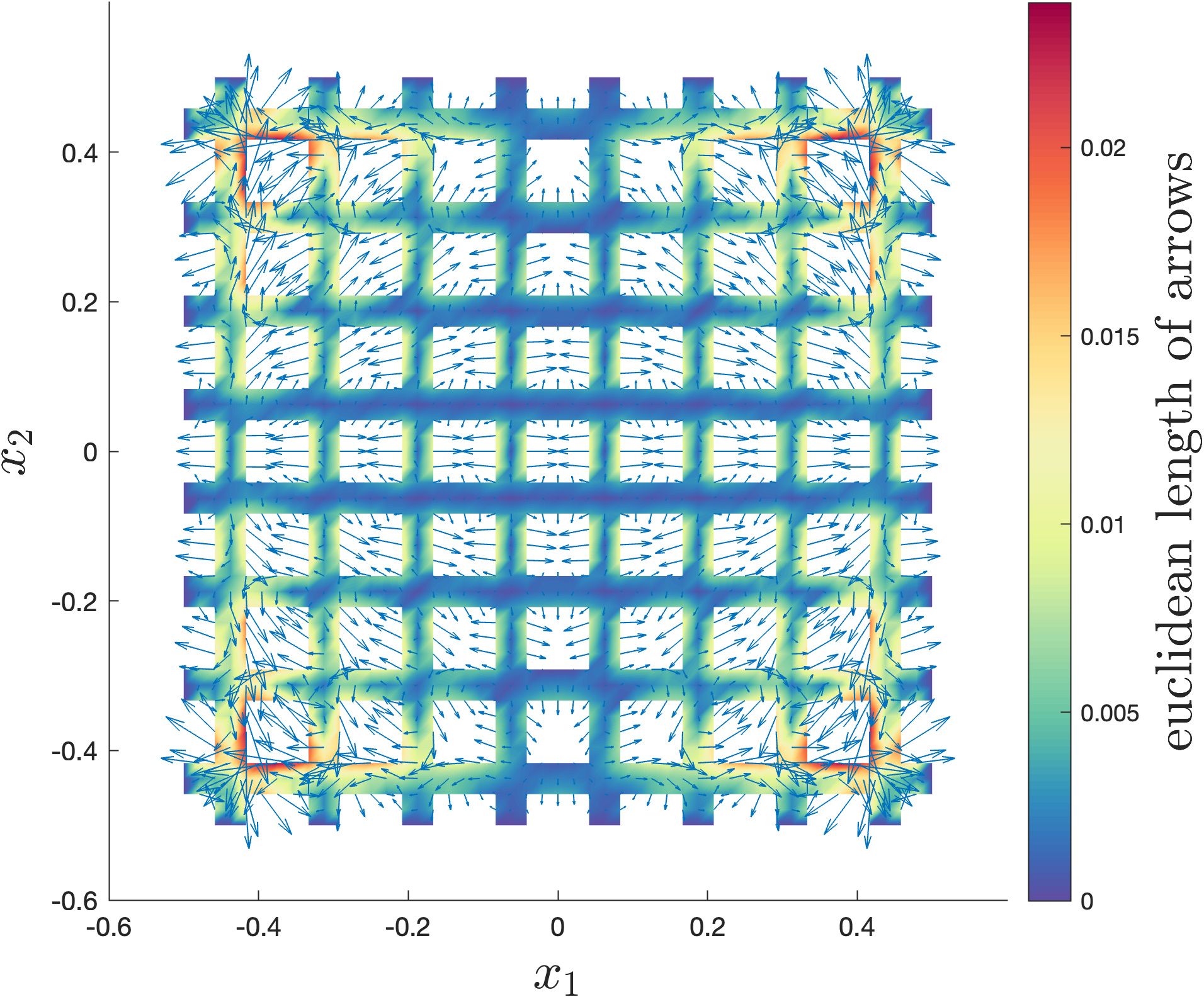}
\caption{$\bfue(t_\text{eval},\cdot)-\bfu\arrowvert_{\Oes}(t_\text{eval},\cdot)$}
\label{fig:corrector_plots:elasticity:c}
\end{subfigure}
\hfill
\begin{subfigure}[b]{0.49\textwidth}
\centering
\includegraphics[width=0.8\textwidth]{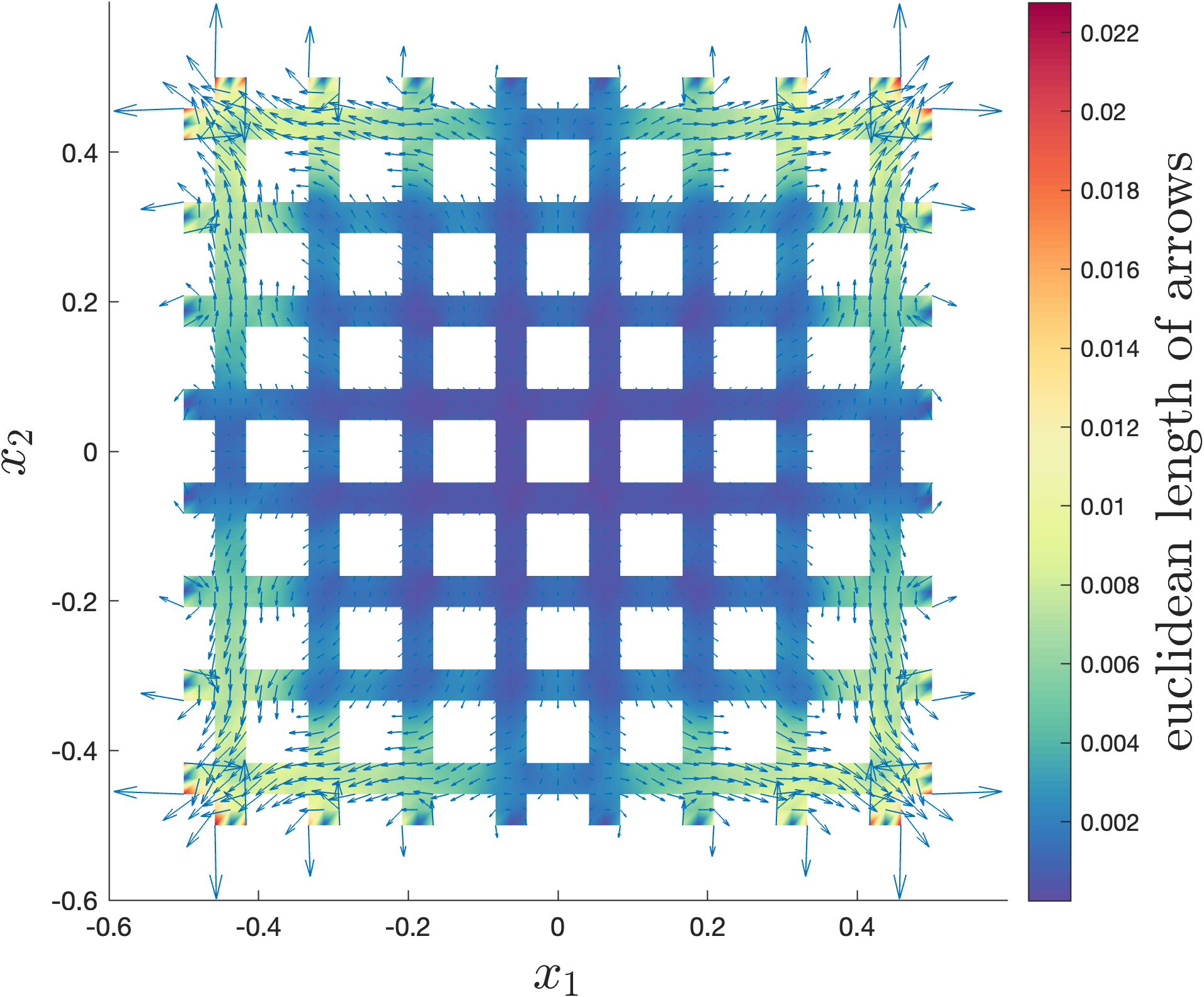}
\caption{$\bfue(t_\text{eval},\cdot)-(\bfu\arrowvert_{\Oes}(t_\text{eval},\cdot) + \eps \bfu_1(t_\text{eval},\cdot, \frac{\cdot}{\eps}))$}
\label{fig:corrector_plots:elasticity:d}
\end{subfigure}
\hfill
\begin{subfigure}[b]{0.49\textwidth}
\centering
\includegraphics[width=0.8\textwidth]{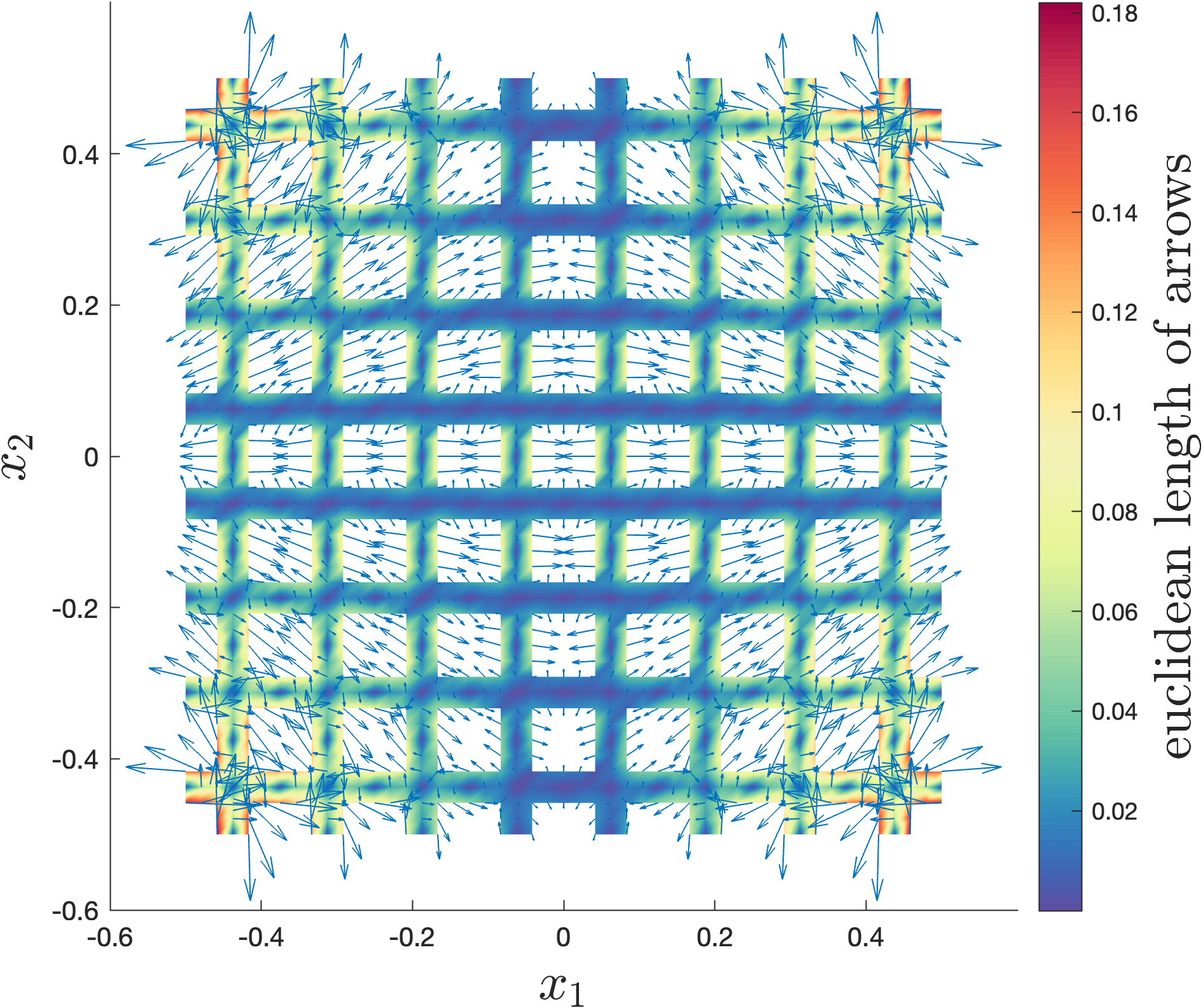}
\caption{$\bfu_1(t_\text{eval},\cdot,\frac{\cdot}{\eps})$}
\label{fig:corrector_plots:elasticity:e}
\end{subfigure}
\caption{Spatial visualization of the microscopic displacement, effective displacement, first displacement corrector and errors, obtained from simulations for $\eps=\frac{1}{8}$ at $t_\text{eval}=0.5$. Note that the arrows in Figures (c)-(e) are scaled by an appropriate factor for a better visibility. The displacement magnitude, i.e. the euclidean norm of the arrows, is color-encoded on the domain.}
\label{fig:corrector_plots:elasticity}
\end{figure}%
\begin{figure}
\centering
\begin{subfigure}[b]{0.49\textwidth}
\centering
\includegraphics[width=0.8\textwidth]{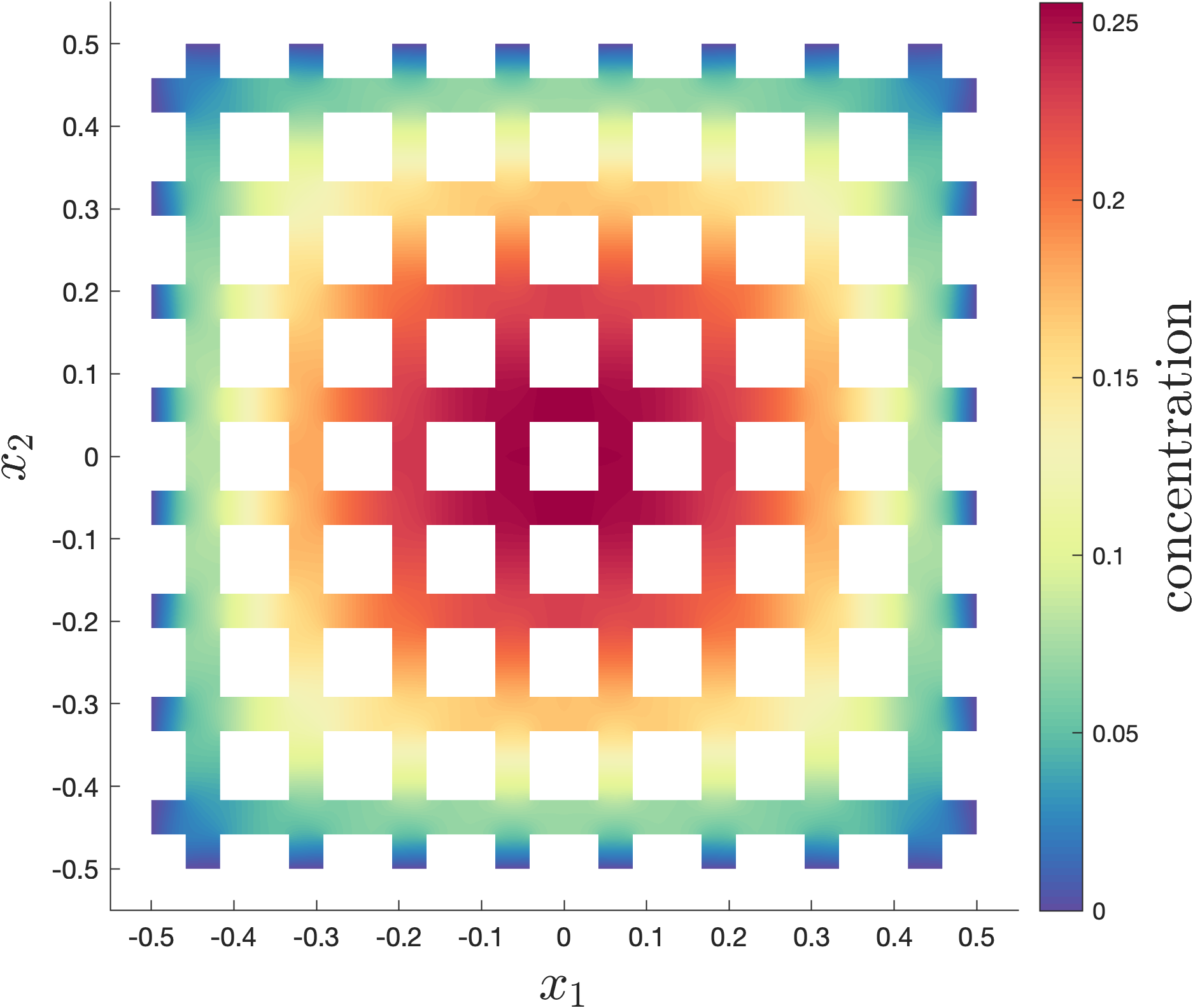}
\caption{$c_\eps(t_\text{eval},\cdot)$}
\label{fig:corrector_plots:diffusion:a}
\end{subfigure}
\hfill
\begin{subfigure}[b]{0.49\textwidth}
\centering
\includegraphics[width=0.8\textwidth]{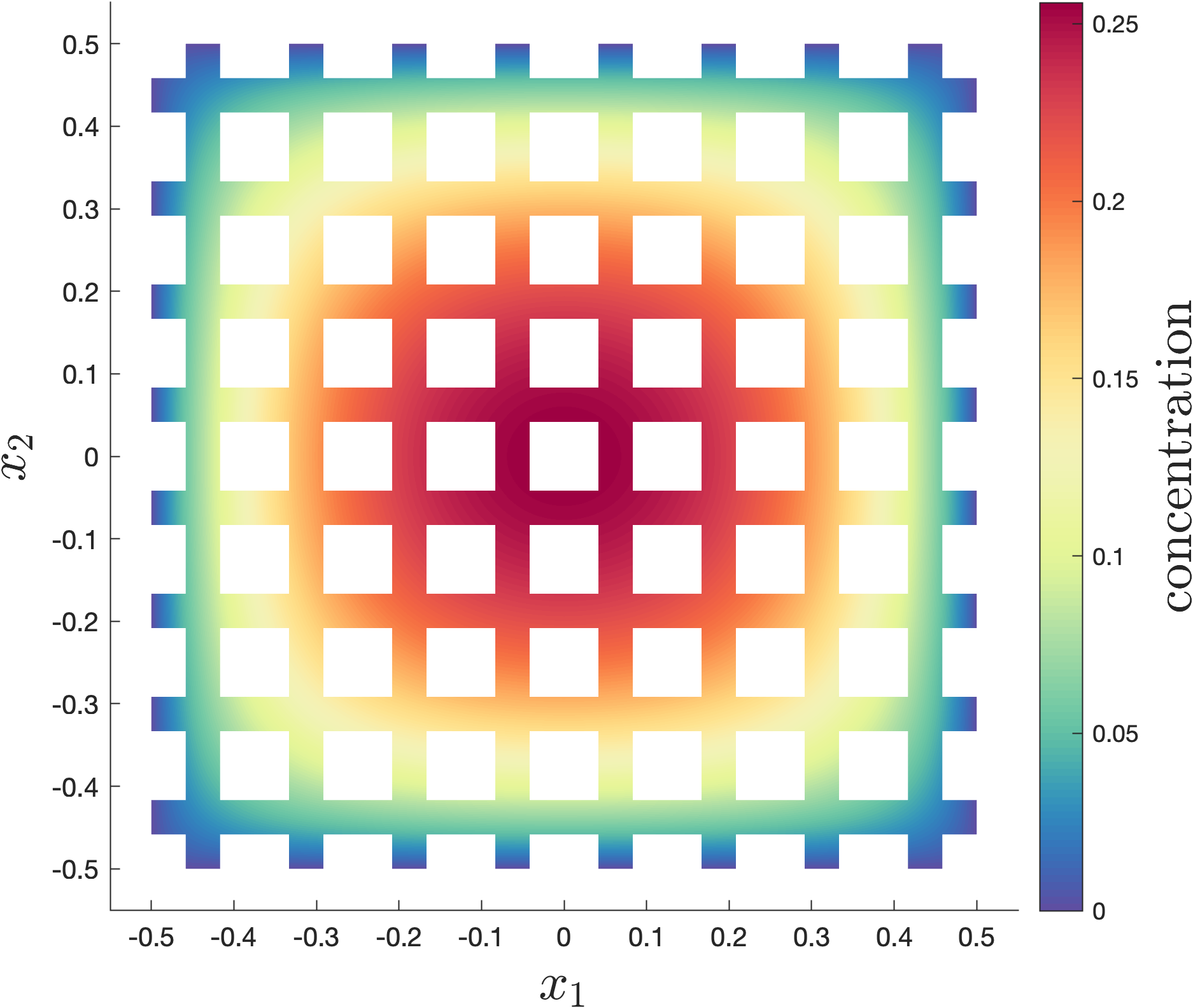}
\caption{$c\arrowvert_{\Oes}(t_\text{eval},\cdot)$}
\label{fig:corrector_plots:diffusion:b}
\end{subfigure}
\hfill
\begin{subfigure}[b]{0.49\textwidth}
\centering
\includegraphics[width=0.8\textwidth]{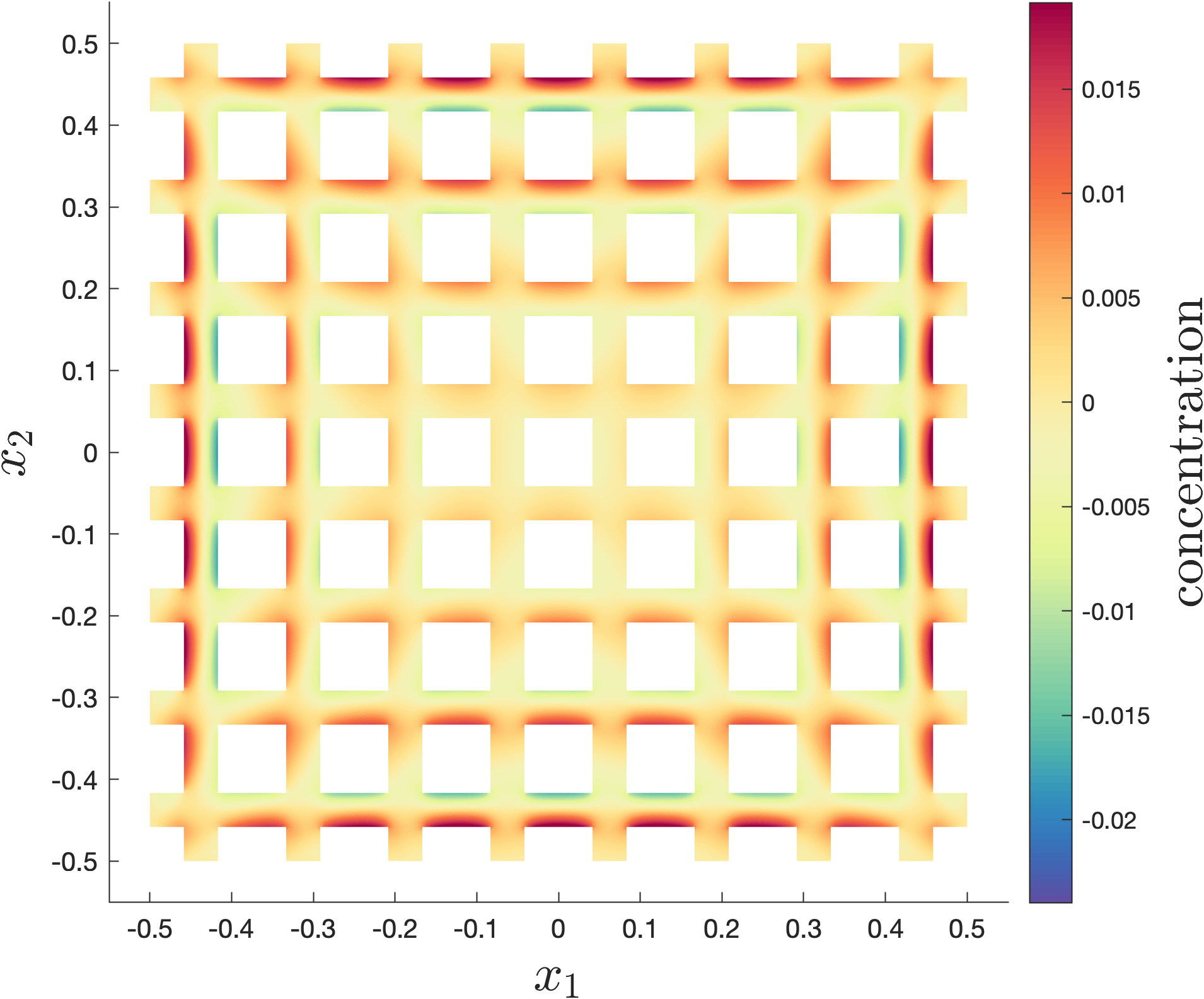}
\caption{$c_\eps(t_\text{eval},\cdot)-c\arrowvert_{\Oes}(t_\text{eval},\cdot) $}
\label{fig:corrector_plots:diffusion:c}
\end{subfigure}
\hfill
\begin{subfigure}[b]{0.49\textwidth}
\centering
\includegraphics[width=0.8\textwidth]{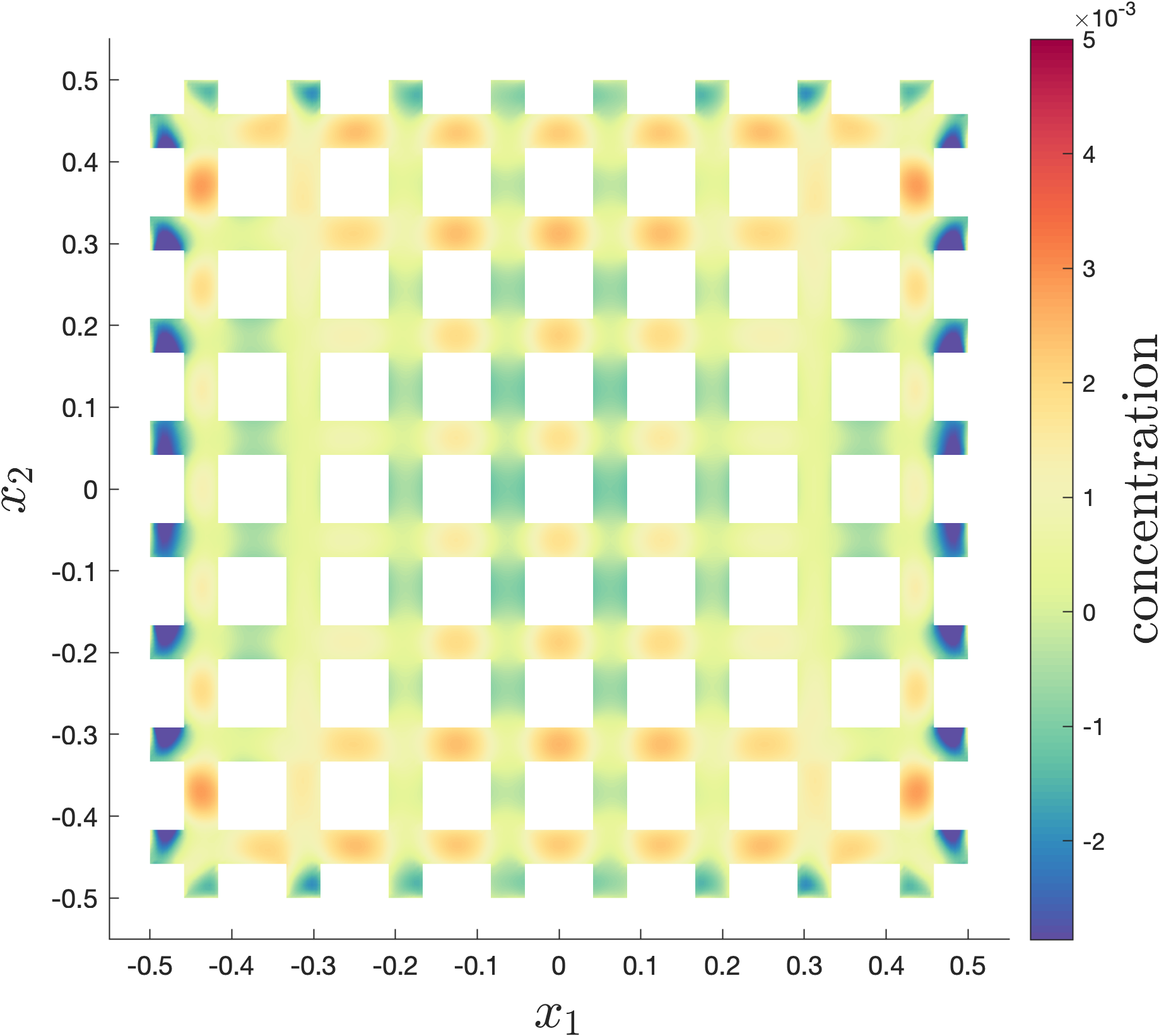}
\caption{$c_\eps(t_\text{eval},\cdot) - (c\arrowvert_{\Oes}(t_\text{eval},\cdot)+ \eps c_1(t_\text{eval},\cdot, \frac{\cdot}{\eps}))$}
\label{fig:corrector_plots:diffusion:d}
\end{subfigure}
\hfill
\begin{subfigure}[b]{0.49\textwidth}
\centering
\includegraphics[width=0.8\textwidth]{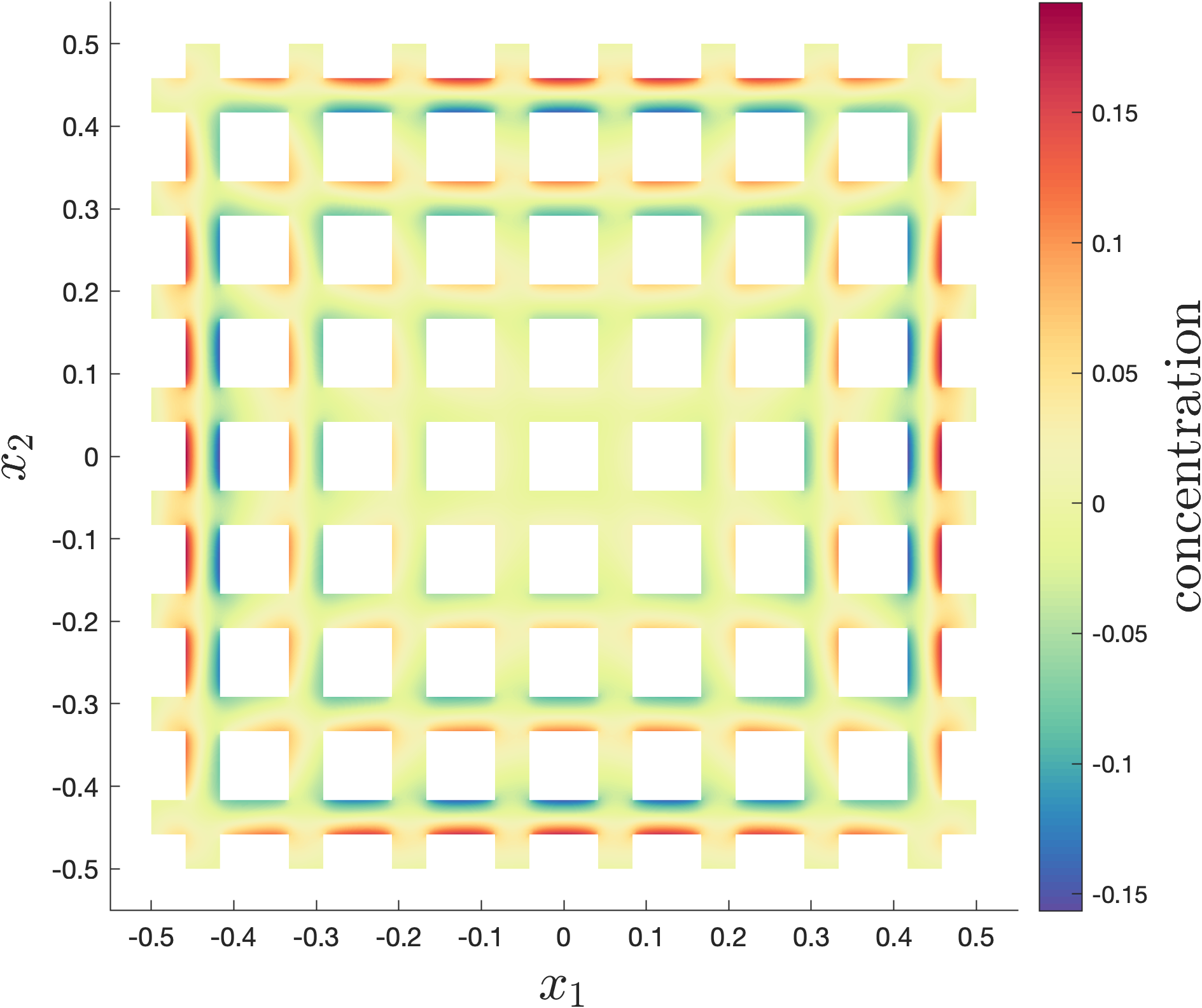}
\caption{$c_1(t_\text{eval},\cdot, \frac{\cdot}{\eps})$}
\label{fig:corrector_plots:diffusion:e}
\end{subfigure}
\caption{Spatial visualization of the microscopic concentration, effective concentration, first concentration corrector and errors, obtained from simulations for $\eps=\frac{1}{8}$ at $t_\text{eval}=0.5$.}
\label{fig:corrector_plots:diffusion}
\end{figure}%
In Figure \ref{fig:corrector_plots:elasticity:a} and \ref{fig:corrector_plots:elasticity:b}, we visualize the solutions $\bfue$ and $\bfu$ to the microscopic problem and the effective micro-macro problem, for a fixed time point $t_\text{eval} = 0.5$. Although there is no apparent difference between $\bfu\arrowvert_{\Oes}$ and $\bfue$, there is a mismatch between these functions as can be seen in Figure \ref{fig:corrector_plots:elasticity:c} where $\bfue-\bfu\arrowvert_{\Oes}$ is visualized. The largest error is found near the ”macroscopic corners” of the domain, but also in the inner region an error is visible. Adding the scaled corrector term $\eps \bfu_1$, see Figure \ref{fig:corrector_plots:elasticity:d},  improves the approximation inside the domain significantly and also to some extent in the corners, as can be seen from Figure  \ref{fig:corrector_plots:elasticity:d}. However, a reduced mismatch in the corners still remains and at the "macroscopic" edges of the domain the approximation is worsened. Here, the well known boundary layer effect is visible, which occurs due to the \textit{Dirichlet} boundary conditions imposed to $\bfue$ and $\bfu$, but which are not fulfilled by the corrector term $\bfu_1$. \par 
Analogous results for the diffusion problem are displayed in Figure \ref{fig:corrector_plots:diffusion}. Here, again addition of the scaled corrector term $\eps c_1$ improves the approximation in the interior of the domain, while inducing an error at the outer part of the boundary.  This error is more pronounced at the lateral part of the boundary, where the deformation gradient has higher values. \par
\section{Comparison of the effective micro-macro model with alternative models for transport problems in elastic perforated media}
\begin{figure}
\centering
\includegraphics[width = 0.75\textwidth]{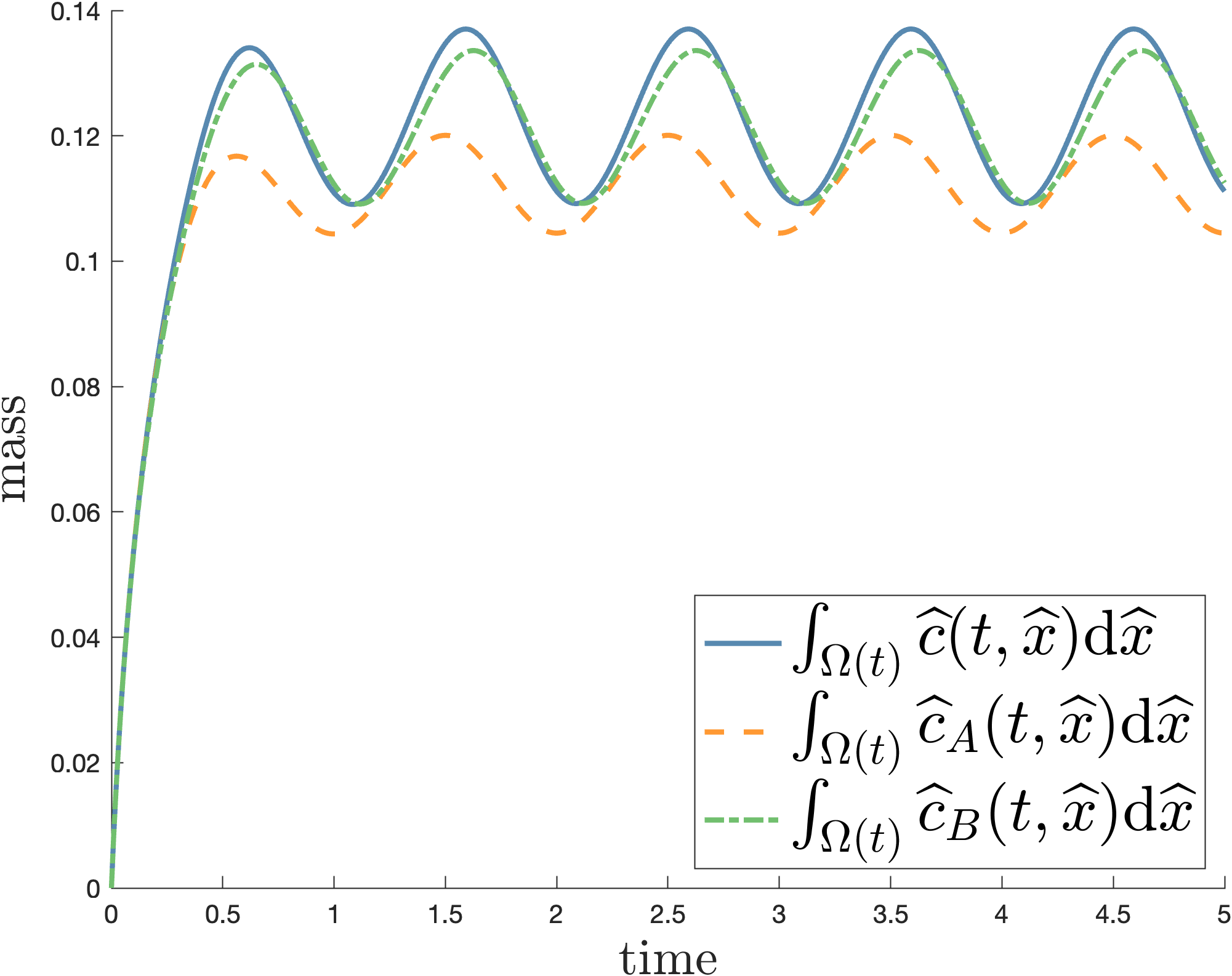}
\caption{Current mass over time according to different descriptions of transport processes in elastically deformable perforated media.}
\label{fig:alternative_models_mass}
\end{figure}
The effective micro-macro problem 
\eqref{eq:effective_micro_macro_model}-\eqref{def:F_0} was derived in \cite{knoch2023} from a microscopic problem in mixed \textit{Eulerian}/\textit{Lagrangian} formulation by first transforming the microscopic transport problem onto the fixed reference domain and then formally homogenizing the problem on the fixed domain. As a result, the effective model has a strongly coupled micro-macro structure, requiring a high numerical effort for the simulations.\par
In this section, we discuss two alternative approaches for the derivation of effective models for transport processes in elastic perforated media, which lead to "simpler" models with a reduced numerical complexity. Then, we compare their solutions with the solution of the effective micro-macro model for the simulation scenario presented at the beginning of the previous section.\par 
The first approach, referred to as \textit{alternative approach A}, is to formulate the microscopic elasticity-transport problem directly in the unified \textit{Lagrangian} framework, i.e. to formulate the microscopic transport problem on the reference domain $\Oes$. Previously, this approach has been used, e.g., in \cite{jaeger2011}, where transport inside elastically deformable tissue cells has been coupled to transport in the extracellular space. In the microscopic model, the fluid-solid-interface has been linearized, resulting in transport equations on the reference fluid and solid domain  respectively. In \cite{bukac2015}, transport processes in elastically deformable domains have been considered in the reference domain and then, for the visualization, a push forward of the concentration onto the deformed domain was considered. 
Thus, following \textit{approach A}, i.e., starting from the transport equation in the reference domain, the method of asymptotic expansion directly leads to the classical effective model for diffusion problems in fixed domains
\begin{equation} \label{eq:alternative_A}
|Y^s| \partial_t c_A - \nabla \cdot \left( \Dhom_A\nabla c_A \right) = |Y^s|f_d \quad \mbox{ in } (0,T) \times \Omega.
\end{equation}
Here, the entries of the constant effective diffusion tensor $\Dhom_A \in \R^{n \times n}$ are given by 
$$
\Dhom_{A,ij} = \int_{Y^s} \wh{\textbf{D}} \left( \textbf{e}_i + \nabla_y\eta_i(y) \right) \cdot \left( \textbf{e}_j+ \nabla_y\eta_j(y) \right) \textrm{d}y.
$$
 The cell problems are independent of $t$ and $x$:
\begin{align*}
-\nabla_y \cdot \left[ \wh{\bfD} \left( \textbf{e}_i + \nabla_y \eta_i(y) \right) \right] &= 0 & \mbox{ in }  Y^s, \\
- \wh{\bfD} \left( \textbf{e}_i + \nabla_y \eta_i(y) \right) \cdot \textbf{n}_\Gamma &= 0 & \mbox{ on }  \Gamma,\\
\eta_i \mbox{ is } Y^s \mbox{-periodic in $y$, } \, \int_{Y^s} \eta_i \textrm{d}y = 0.
\end{align*}
In this case, the computation of the effective diffusion coefficient $\Dhom_A$ can be done once and for all at the beginning of the simulation, analogously to the computation of the effective elasticity tensor $\Ahom$ in Algorithm \ref{alg:1} As a result, the computational effort is significantly reduced in comparison to the effective micro-macro model \eqref{eq:effective_micro_macro_model}-\eqref{def:F_0}. For the graphical presentation of the simulation results the concentration is transformed to the current deformed domain via the push-forward
$$
\wh{c}_A(t,\wh{x}) := c_A(t, \bfS^{-1}(t,\wh{x})),
$$
using the macroscopic deformation $\bfS(t,x):=x+\bfu(t,x)$.
\par 
A second approach, referred to as \textit{alternative approach B}, consists of using the constant effective coefficients from the previous approach A (computed on the reference domain), including the effective diffusion tensor $\Dhom_A$, to formulate a homogenized transport problem in the \textit{Eulerian} framework: 
\begin{equation}\label{eq:approach_B}
|Y^s|\partial_t \wh{c}_B - \wh{\nabla} \cdot \left( \wh{\bfv}_B\wh{c}_B - \Dhom_A\wh{\nabla}\wh{c}_B \right) = |Y^s| \wh{f}_d \quad \mbox{ in } Q^T_B,
\end{equation}
where $Q^T_B := \bigcup_{t\in(0,T)} \{t\} \times \Omega(t)$, $\Om(t):=\{ \wh{x} \in \R^n \mid \wh{x} = \bfS(t,x),\; x \in \Om \}$ and $\wh{\bfv}_B(t,\wh{x}):= \partial_t \bfS(t, \bfS^{-1}(t,\wh{x}))$. Then, in order to perform the simulations by using our numerical framework, the problem \eqref{eq:approach_B} is transformed to the fixed reference domain $\Omega$. The transformed problem reads:
\begin{equation}\label{eq:alternative_B}
|Y^s|\partial_t(J_B c_B) - \nabla \cdot \left( \Dhom_B \nabla c_B \right) = |Y^s|J_B f_d \quad \mbox{ in } (0,T) \times \Omega,
\end{equation}
with
\begin{align*}
c_B(t,x) &:= \wh{c}_B(t,\bfS(t,x)),\\
\Dhom_B(t,x) &:= \left[J_B \bfF_B^{-1} \Dhom_A \bfF^{-T}_B\right](t,x),\\
\bfF_B(t,x) &:= \nabla \bfS(t,x),\\
J_B(t,x) &:= \det(\bfF_B(t,x)).\\
\end{align*}%
Finally, 
the solution on the time-dependent domain is given by
$$
\wh{c}_B(t,\wh{x}) = c_B(t, \bfS^{-1}(t,\wh{x})).
$$
This second approach is highly phenomenological, however it also has a reduced complexity compared to the effective micro-macro model. \par 
Now, we compare the solutions of the two alternative approaches described above with the solution of the effective micro-macro model numerically. We perform simulations, using the simulation scenario from Section \ref{sec:numerical_justification}, and visualize the mass over time on the deformed domain, see Figure \ref{fig:alternative_models_mass}. We observe that all curves exhibit in some way the time-dependent oscillations coming from the stretching-relaxing cycles of the domain. However, neither of the solutions to the alternative approaches match the results obtained from the effective micro-macro model \eqref{eq:effective_micro_macro_model}-\eqref{def:F_0} which is depicted in blue (solid line). In the case of \textit{alternative approach A} (orange, dashed line), the time-oscillations are shifted and the mass lies well below the mass in of the other approaches. The latter observation can be attributed to the constant source term $|Y^s|f_d$ on the right-hand side of Equation \eqref{eq:alternative_A}, see also \eqref{eq:rhs_and_init_value}, which does not take into account the change of source due to the deformation. In the case of \textit{alternative approach B} (green, dash-dotted line) the time-oscillations are still too damped and slightly shifted in comparison to the blue (solid) curve. However, the mass is significantly increased in comparison to the orange (dashed) curve, since the right-hand side of Equation \eqref{eq:alternative_B} does account for changes in the source term due to deformations. Altoghether, we see from Figure \ref{fig:alternative_models_mass} that is necessary to consider the effective micro-macro problem \eqref{eq:effective_micro_macro_model}-\eqref{def:F_0} in order to capture important features of the transport in elastic heterogeneous media, like the effect of the deforming microstructure on the amplitude of the time-oscillations of the mass in our numerical experiment with cyclic stretching and relaxing.
%
\section{Conclusion and outlook}
\label{sec:discussion}

In this paper, we proved global existence in time and uniqueness of solutions to an effective micro-macro problem for reactive transport in elastically derformable porous media,  which was formally derived starting from a microscopic model in mixed \textit{Lagrangian/Eulerian} formulation. This result thus generalizes existing results on global well-posedness for elasticity-transport models of micro-macro type to situations for which no additional assumptions (such as linearization) are made about the evolution of the free boundary at the microscale. 
The numerical validation of the effective model underpins its use for the description and investigation of various applications in deformable perforated media.  \par 
At the same time, the analytical and numerical investigations in this paper are a starting point for future research on transport processes in elastically deformable porous media.   
For this purpose, processes in the fluid part of the porous medium have to be considered. As an application, we mention the transport of bio-chemical compounds, like glucose or lactate, in biological tissue which takes place in both, the intracellular and the extracellular space.
Such models, however, involve transport equations coupled to nonlinear fluid-structure interactions in complex microscopic geometries. The analysis of the resulting micro-macro models is an open problem so far.  It is particularly interesting to find and interpret assumptions that guarantee the well-posedness of such models. Furthermore, their numerical treatment require the extension of the computational tools used in this paper to micro-macro models for fluid-elasticity-transport systems.  
%
\bibliographystyle{siamplain}
\bibliography{literature}

\begin{thebibliography}{10}

\bibitem{allaire2009}
{\sc G.~Allaire and K.~El~Ganaoui}, {\em Homogenization of a conductive and
  radiative heat transfer problem}, Multiscale Modeling \& Simulation, 7
  (2009), pp.~1148--1170.

\bibitem{alt2016}
{\sc H.~W. Alt}, {\em Linear functional analysis}, Springer, 2016.

\bibitem{dealii2023}
{\sc D.~Arndt, W.~Bangerth, M.~Bergbauer, M.~Feder, M.~Fehling, J.~Heinz,
  T.~Heister, L.~Heltai, M.~Kronbichler, M.~Maier, P.~Munch, J.-P. Pelteret,
  B.~Turcksin, D.~Wells, and S.~Zampini}, {\em The \texttt{deal.II} library,
  version 9.5}, Journal of Numerical Mathematics,  (2023, submitted),
  \url{https://dealii.org/deal95-preprint.pdf}.

\bibitem{dealii2019}
{\sc D.~Arndt, W.~Bangerth, D.~Davydov, T.~Heister, L.~Heltai, M.~Kronbichler,
  M.~Maier, J.-P. Pelteret, B.~Turcksin, and D.~Wells}, {\em The {deal.II}
  finite element library: Design, features, and insights}, Computers \&
  Mathematics with Applications, 81 (2021), pp.~407--422,
  \url{https://doi.org/10.1016/j.camwa.2020.02.022},
  \url{https://arxiv.org/abs/1910.13247}.

\bibitem{bakhvalov1984}
{\sc N.~Bakhvalov and G.~Panasenko}, {\em Homogenization: averaging processes
  in periodic media}, Nauka, Moscow (in Russian). English translation in:
  mathematics and its applications (Soviet Series) 36,  (1984).

\bibitem{bensoussan2011}
{\sc A.~Bensoussan, J.-L. Lions, and G.~Papanicolaou}, {\em Asymptotic analysis
  for periodic structures}, vol.~374, American Mathematical Soc., 2011.

\bibitem{biot1956}
{\sc M.~A. Biot}, {\em Theory of propagation of elastic waves in a
  fluid-saturated porous solid. {I}. {L}ower frequency range, and {II}. higher
  frequency range}, Jour. Acoustic Soc. Amer., 28 (1956), pp.~168--178 and
  179--191.

\bibitem{brown2014}
{\sc D.~L. Brown, P.~Popov, and Y.~Efendiev}, {\em Effective equations for
  fluid-structure interaction with applications to poroelasticity}, Applicable
  Analysis, 93 (2014), pp.~771--790.

\bibitem{brun2018}
{\sc M.~K. Brun, I.~Berre, J.~M. Nordbotten, and F.~A. Radu}, {\em Upscaling of
  the coupling of hydromechanical and thermal processes in a quasi-static
  poroelastic medium}, Transport in Porous Media, 124 (2018), pp.~137--158.

\bibitem{bukac2015}
{\sc M.~Buka{\v{c}}, I.~Yotov, R.~Zakerzadeh, and P.~Zunino}, {\em Partitioning
  strategies for the interaction of a fluid with a poroelastic material based
  on a nitsche’s coupling approach}, Computer Methods in Applied Mechanics
  and Engineering, 292 (2015), pp.~138--170.

\bibitem{ciarlet1988}
{\sc P.~G. Ciarlet}, {\em Three-dimensional elasticity}, Elsevier, 1988.

\bibitem{cioranescu1999}
{\sc D.~Cioranescu and P.~Donato}, {\em An introduction to homogenization},
  vol.~17, Oxford university press Oxford, 1999.

\bibitem{clopeau2001}
{\sc T.~Clopeau, J.~Ferrin, R.~Gilbert, and A.~Mikeli{\'c}}, {\em Homogenizing
  the acoustic properties of the seabed, part ii}, Mathematical and Computer
  Modelling, 33 (2001), pp.~821--841.

\bibitem{collis2017}
{\sc J.~Collis, D.~Brown, M.~E. Hubbard, and R.~D. O’Dea}, {\em Effective
  equations governing an active poroelastic medium}, Proceedings of the Royal
  Society A: Mathematical, Physical and Engineering Sciences, 473 (2017),
  p.~20160755.

\bibitem{eden2017}
{\sc M.~Eden and A.~Muntean}, {\em Homogenization of a fully coupled
  thermoelasticity problem for a highly heterogeneous medium with a priori
  known phase transformations}, Mathematical Methods in the Applied Sciences,
  40 (2017), pp.~3955--3972.

\bibitem{evans2010}
{\sc L.~C. Evans}, {\em Partial differential equations}, vol.~19, American
  Mathematical Soc., 2010.

\bibitem{gahn2021}
{\sc M.~Gahn, M.~Neuss-Radu, and I.~S. Pop}, {\em Homogenization of a
  reaction-diffusion-advection problem in an evolving micro-domain and
  including nonlinear boundary conditions}, Journal of Differential Equations,
  289 (2021), pp.~95--127.

\bibitem{gahn2023}
{\sc M.~Gahn and I.~S. Pop}, {\em Homogenization of a mineral dissolution and
  precipitation model involving free boundaries at the micro scale}, Journal of
  Differential Equations, 343 (2023), pp.~90--151.

\bibitem{giaquinta2013}
{\sc M.~Giaquinta and L.~Martinazzi}, {\em An introduction to the regularity
  theory for elliptic systems, harmonic maps and minimal graphs}, Springer
  Science \& Business Media, 2013.

\bibitem{gilbert2000}
{\sc R.~Gilbert and A.~Mikeli{\'c}}, {\em Homogenizing the acoustic properties
  of the seabed: Part i}, Nonlinear Analysis: Theory, Methods \& Applications,
  40 (2000), pp.~185--212.

\bibitem{horn2012}
{\sc R.~A. Horn and C.~R. Johnson}, {\em Matrix analysis}, Cambridge university
  press, 2012.

\bibitem{jaeger2009_2}
{\sc W.~J{\"a}ger, A.~Mikeli{\'c}, and M.~Neuss-Radu}, {\em Analysis of
  differential equations modelling the reactive flow through a deformable
  system of cells}, Archive for rational mechanics and analysis, 192 (2009),
  pp.~331--374.

\bibitem{jaeger2011}
{\sc W.~J{\"a}ger, A.~Mikeli{\'c}, and M.~Neuss-Radu}, {\em Homogenization
  limit of a model system for interaction of flow, chemical reactions, and
  mechanics in cell tissues}, SIAM journal on mathematical analysis, 43 (2011),
  pp.~1390--1435.

\bibitem{khoa2021}
{\sc V.~A. Khoa, T.~Thoa~Thieu, and E.~R. Ijioma}, {\em On a pore-scale
  stationary diffusion equation: Scaling effects and correctors for the
  homogenization limit}, Discrete and continuous dynamical systems. Series B,
  26 (2021), pp.~2451--2477.

\bibitem{knoch2023}
{\sc J.~Knoch, M.~Gahn, M.~Neuss-Radu, and N.~Neu{\ss}}, {\em Multi-scale
  modeling and simulation of transport processes in an elastically deformable
  perforated medium}, Transport in Porous Media, 147 (2023), pp.~93--123.

\bibitem{oleinik1992}
{\sc O.~Oleinik, G.~Yosifyan, and A.~Shamaev}, {\em Mathematical Problems in
  Elasticity and Homogenization}, North-Holland, 1992.

\bibitem{olivares2021}
{\sc M.~B. Olivares, C.~Bringedal, and I.~S. Pop}, {\em A two-scale iterative
  scheme for a phase-field model for precipitation and dissolution in porous
  media}, Applied Mathematics and Computation, 396 (2021), p.~125933.

\bibitem{peter2007}
{\sc M.~A. Peter}, {\em Homogenisation in domains with evolving
  microstructure}, Comptes Rendus M{\'e}canique, 335 (2007), pp.~357--362.

\bibitem{ray2019}
{\sc N.~Ray, J.~Oberlander, and P.~Frolkovic}, {\em Numerical investigation of
  a fully coupled micro-macro model for mineral dissolution and precipitation},
  Computational Geosciences, 23 (2019), pp.~1173--1192.

\bibitem{rohan2020}
{\sc E.~Rohan and S.~Naili}, {\em Homogenization of the fluid--structure
  interaction in acoustics of porous media perfused by viscous fluid},
  Zeitschrift f{\"u}r angewandte Mathematik und Physik, 71 (2020), pp.~1--28.

\bibitem{sanchez-palencia1980}
{\sc E.~S{\'a}nchez-Palencia}, {\em Non-homogeneous media and vibration
  theory}, Lecture Note in Physics, Springer-Verlag, 320 (1980), pp.~57--65.

\bibitem{ye2021}
{\sc S.~Ye, Q.~Ma, B.~Hu, J.~Cui, and X.~Jiang}, {\em Multiscale asymptotic
  analysis and computations for steklov eigenvalue problem in periodically
  perforated domain}, Mathematical Methods in the Applied Sciences, 44 (2021),
  pp.~12592--12612.

\end{thebibliography}
\end{document}